\documentclass[smallextended]{svjour3}       % onecolumn (second format)
\smartqed  % flush right qed marks, e.g. at end of proof
\usepackage{graphicx}

\usepackage{latexsym}
\usepackage{amssymb,amsbsy,amsmath,amsfonts,amssymb,amscd,mathtools,amsfonts,mathrsfs}

\usepackage{graphicx}
\usepackage{dsfont}
\usepackage{subfig}
\usepackage{pdfsync}
\usepackage{color} 
\usepackage{float}

\newcommand {\supp}   {\text{{\rm supp}}}
\newcommand{\argmax}[1]{\underset{#1}{\operatorname{arg}\,\operatorname{max}}\;}

\newcommand {\p}   {\partial}

\newcommand {\e}  {\varepsilon}

\newcommand{\beq}{\begin{equation}}
\newcommand{\eeq}{\end{equation}}
\newcommand{\bea} {\begin{array}{rl}}
\newcommand{\eea} {\end{array}}
\newcommand{\bepa}{\left\{ \begin{array}{l}}
\newcommand{\eepa} {\end{array}\right.}

\newtheorem{cor}[theorem]{Corollary}

\usepackage{natbib}
\setcitestyle{authoryear,open={(},close={)}}
\begin{document}
\title{A structured population model of clonal selection in acute leukemias with multiple maturation stages
}

\titlerunning{A continuously-structured model of clonal selection in acute leukemias}        % if too long for running head

\author{Tommaso Lorenzi \and Anna Marciniak-Czochra \and Thomas Stiehl 
}


\institute{Tommaso Lorenzi \at
		School of Mathematics and Statistics, University of St Andrews, North Haugh, St Andrews, Fife, KY16 9SS, United Kingdom\\
              \email{tl47@st-andrews.ac.uk}           %  \\
                      \and
           Anna Marciniak-Czochra \at
		Institute of Applied Mathematics, BIOQUANT and IWR, Im Neuenheimer Feld 205, Heidelberg University, 69120 Heidelberg, Germany\\
		\email{Anna.Marciniak@iwr.uni-heidelberg.de}           %  \\
		\and
           Thomas Stiehl \at
		Institute of Applied Mathematics, Im Neuenheimer Feld 205, Heidelberg University, 69120 Heidelberg, Germany\\
		 \email{thomas.stiehl@iwr.uni-heidelberg.de}           %  \\
}

\date{Received: date / Accepted: date}

\maketitle

\begin{abstract}
Recent progress in genetic techniques has shed light on the complex co-evolution of malignant cell clones in leukemias. However, several aspects of clonal selection still remain unclear. In this paper, we present a multi-compartmental continuously structured population model of selection dynamics in acute leukemias, which consists of a system of coupled integro-differential equations. Our model can be analysed in a more efficient way than classical models formulated in terms of ordinary differential equations. Exploiting the analytical tractability of this model, we investigate how clonal selection is shaped by the self-renewal fraction and the proliferation rate of leukemic cells at different maturation stages. We integrate analytical results with numerical solutions of a calibrated version of the model based on real patient data. In summary, our mathematical results formalise the biological notion that clonal selection is driven by the self-renewal fraction of leukemic stem cells and the clones that possess the highest value of this parameter are ultimately selected. Moreover, we demonstrate that the self-renewal fraction and the proliferation rate of non-stem cells do not have a substantial impact on clonal selection. Taken together, our results indicate that interclonal variability in the self-renewal fraction of leukemic stem cells provides the necessary substrate for clonal selection to act upon.

\end{abstract}

\section{Introduction}
\noindent {\bf Statement of the biological problem.} Leukemias are malignant diseases of the hematopoietic system. Similarly to the healthy hematopoietic system, the leukemic cell bulk is organised as a hierarchy of multiple maturation stages (\emph{i.e.} maturation compartments) -- from stem cells through a number of increasingly mature progenitor cells to the most mature cells~\citep{Bonnet,Hope}. Red blood cells, white blood cells and platelets constitute the most mature compartment of healthy cells, whereas non-functional leukemic blasts are the most mature leukemic cells. The extensive growth of non-functional leukemic blasts leads to impaired hematopoiesis and causes a shortage of healthy blood cells. 

In contrast to the most mature cells, which do not divide and die at a constant rate, stem and progenitor cells can proliferate and give rise to progeny cells which are either at the same maturation stage as their parent (self-renewal) or at a subsequent maturation stage (differentiation). These processes can be quantitatively characterised in terms of two parameters: the cell proliferation rate, \emph{i.e.} the number of cell divisions per unit of time, and the cell self-renewal fraction,   \emph{i.e.} the probability that a progeny adopts the same cell fate as its parent~\citep{MarciniakStiehl,StiehlMMNP}.
There is both theoretical~\citep{StiehlBaranHoMarciniakCR} and experimental~\citep{Jung,Metzeler,Wang} evidence suggesting that the proliferation rate and the self-renewal fraction of leukemic cells have a significant impact on the disease dynamics and on patient prognosis~\citep{SMC-Review2017}.

The collection of all stem, progenitor and most mature cells which carry the same set of genetic alterations defines a leukemic clone. Recent experimental evidence indicates that the leukemic cell bulk of an individual patient is composed of multiple clones carrying different mutations~\citep{Anderson,Ding,Ley} and having different functional properties, among which different proliferation rates and self-renewal fractions~\citep{Eppert,Heuser}. Such clonal heterogeneity poses a major obstacle to successful therapy and management of disease relapse~\citep{Choi,VanDelft,LutzHoang,Lutz2013}. In fact, it has been reported that in many cases of acute lymphoblastic leukemias (ALL) and acute myeloid leukemias (AML) relapse is triggered by the selection of clones that have been present as minor (small) clones at the time of diagnosis rather than by newly acquired mutations~\citep{Choi,VanDelft,Ding,Jan}. 

It is becoming increasingly apparent that clonal heterogeneity results from a selection process at the cellular level which triggers the expansion of some clones and the out-competition of others \citep{Belderbos,Choi,Hirsch,Noetzli,StiehlBaranHoMarciniak,Wu}. However, several aspects of clonal selection are so far not well understood and a number of important questions remain open~\citep{SMCReview2019}. In this paper, we focus  on clonal selection taking place before diagnosis or relapse of acute leukemias. We use a mathematical modelling approach to address the following questions:

\begin{itemize}
	\item[{\bf Q1}] What is the role of the proliferation rate and the self-renewal fraction of leukemic cells in clonal selection?
	\item[{\bf Q2}] Can we observe clonal selection among clones that have identical stem-cell properties (in terms of proliferation rate and self-renewal fraction) and differ in their progenitor-cell properties? 
	\item[{\bf Q3}] What are the necessary conditions (in terms of proliferation rate and self-renewal fraction) for the long-term coexistence of different clones?
\end{itemize}

\noindent{\bf Mathematical framework.} A well-established mathematical approach to describing the dynamics of multiple leukemic clones is to use models with multiple compartments formulated in terms of ordinary differential equations (ODEs)~\citep{StiehlBaranHoMarciniak,Werner}. In these models, every cell is characterised by a pair of indices $i=1, \ldots, M$ and $j=0, \ldots, J$. The index $i$ corresponds to the maturation stage of the cell (\emph{i.e.} the compartment to which the cell belongs). The index $j=1, \ldots, J$ indicates what leukemic clone the cell belongs to, whilst the index $j=0$ is conventionally associated to healthy cells. In this modelling framework, a collection of parameters $p^j_i$ and $a^j_i$ is introduced to model, respectively, the proliferation rate and the self-renewal fraction of cells of clone $j=1,\ldots,J$ at the maturation stage $i=1, \ldots, M-1$. Clonal heterogeneity is incorporated into the model by allowing the values of these parameters to change from one clone to the other. The dynamic of every clone is described by a system of $M$ coupled ODEs that track the time evolution of the density of cells at different maturation stages. An additional system of $M$ coupled ODEs is introduced to model the dynamic of the healthy cells. Coherently with biological findings~\citep{Kondo,Layton,Shinjo}, all ODEs of the model are coupled through a feedback signal that regulates the cell dynamics and depends on the total density of cells at the maturation stage $M$ (\emph{i.e.} all the most mature healthy and leukemic cells). A prototypical version of such ODE models can be found in Appendix~A of this paper.

These models consist of systems of $(J+1) \, M$ coupled ODEs. Due to the high degree of clonal heterogeneity usually observed in leukemia patients, the biologically realistic values of $J$ can be very high. As a result, the analysis of these models becomes very hard (if not impossible) in scenarios that are clinically relevant. This poses limitations to the robustness of the biological conclusions that can be obtained using these models. To overcome these limitations, here we present a modelling framework whereby the index $j$ is replaced by a continuous structuring variable $x$. This variable can be seen as a parameterisation of the self-renewal fraction and the proliferation rate of the different clones. Hence, in our modelling framework the parameters $p^j_i$ and $a^j_i$ of the ODE model are replaced by some functions $p_i(x)$ and $a_i(x)$ which represent, respectively, the proliferation rate and the self-renewal fraction of cells that are at the maturation stage $i$ and belong to the clone identified by the variable $x$. This multi-compartmental continuously structured population model consists of a system of $M$ coupled integro-differential equations (IDEs) that can be analysed in a more efficient way than its ODE counterpart.\\

\noindent{\bf Main results of this paper.}
Exploiting the analytical tractability of the model, we address questions {\bf Q1}-{\bf Q3} listed above by elucidating how clonal selection is shaped by the self-renewal fraction and the proliferation rate of leukemic cells at different maturation stages. Analytical results are integrated with numerical solutions of a calibrated version of the model based on patient data from the existing literature. In summary, our results formalise the idea that clonal selection is controlled by the self-renewal fraction of leukemic stem cells and the clones with the highest value of this parameter are ultimately selected. This implies that only the clones whose stem cells are characterised by the highest self-renewal fraction can stably coexist. Finally, our results indicate that interclonal variability in the self-renewal fraction of leukemic stem cells provides the necessary substrate for clonal selection to act upon.

Our analytical work follows earlier papers on the asymptotic analysis of IDEs that arise in the mathematical modelling of selection dynamics in populations structured by physiological traits ~\citep{barles2009concentration,Busse,calsina2013asymptotics,chisholm2016effects,delitala2012asymptotic,desvillettes2008selection,diekmann2005dynamics,lorenzi2014asymptotic,lorz2011dirac,perthame2008dirac,raoul2011long}. In particular, \citet{Busse} have studied a basic version of the model, in which only two maturation stages are considered (\emph{i.e.} $M=2$) and all the clones have the same cell proliferation rate, \emph{i.e.} the functions $p_1(x), \ldots, p_{M-1}(x)$ are constant. The main novelty of our work is that we let the number of maturation stages $M$ be arbitrarily large and we allow the cell proliferation rate to vary from clone to clone, \emph{i.e.} the functions $p_1(x), \ldots, p_{M-1}(x)$ are not necessarily constant. This makes the application domain of our results significantly wider and strengthens the robustness of our biological conclusions. Due to these additional layers of complexity, our analysis builds on a method of proof which is different from that proposed by Busse \emph{et al.} and is based on asymptotic arguments related to those proposed by Desvillettes \emph{et al.} in~\citep{desvillettes2008selection}. We remark that our analytical results rely on general assumptions and, therefore, are applicable to different subtypes of leukemia. 

\section{Description of the model}
\label{sec:model}
We present a multi-compartmental continuously structured population model for the dynamics of cells of multiple leukemic clones and healthy cells at different maturation stages. The model is given by the system of IDEs~\eqref{e.mod3o} and is defined as a continuous version of the multi-compartmental ODE model presented in~\citep{StiehlBaranHoMarciniak} -- for the sake of completeness, we provide a short description of such an ODE model in Appendix~A -- and a generalisation of the continuously-structured population model with two maturation compartments considered in~\citep{Busse}. The key ideas underlying our model are summarised by the schematic diagrams in Fig.~\ref{fig:Model}. 

As illustrated by the scheme in Fig.~\ref{fig:Model}(\textsf{B}), we capture the high degree of clonal heterogeneity usually observed in leukemia patients by introducing a continuous structuring variable $x \in [0,1]$, and we assume that $x=0$ corresponds to healthy cells whereas different leukemic clones are characterised by different values of $x \in (0,1]$. 
\begin{figure}
\centering
\includegraphics[width=9cm]{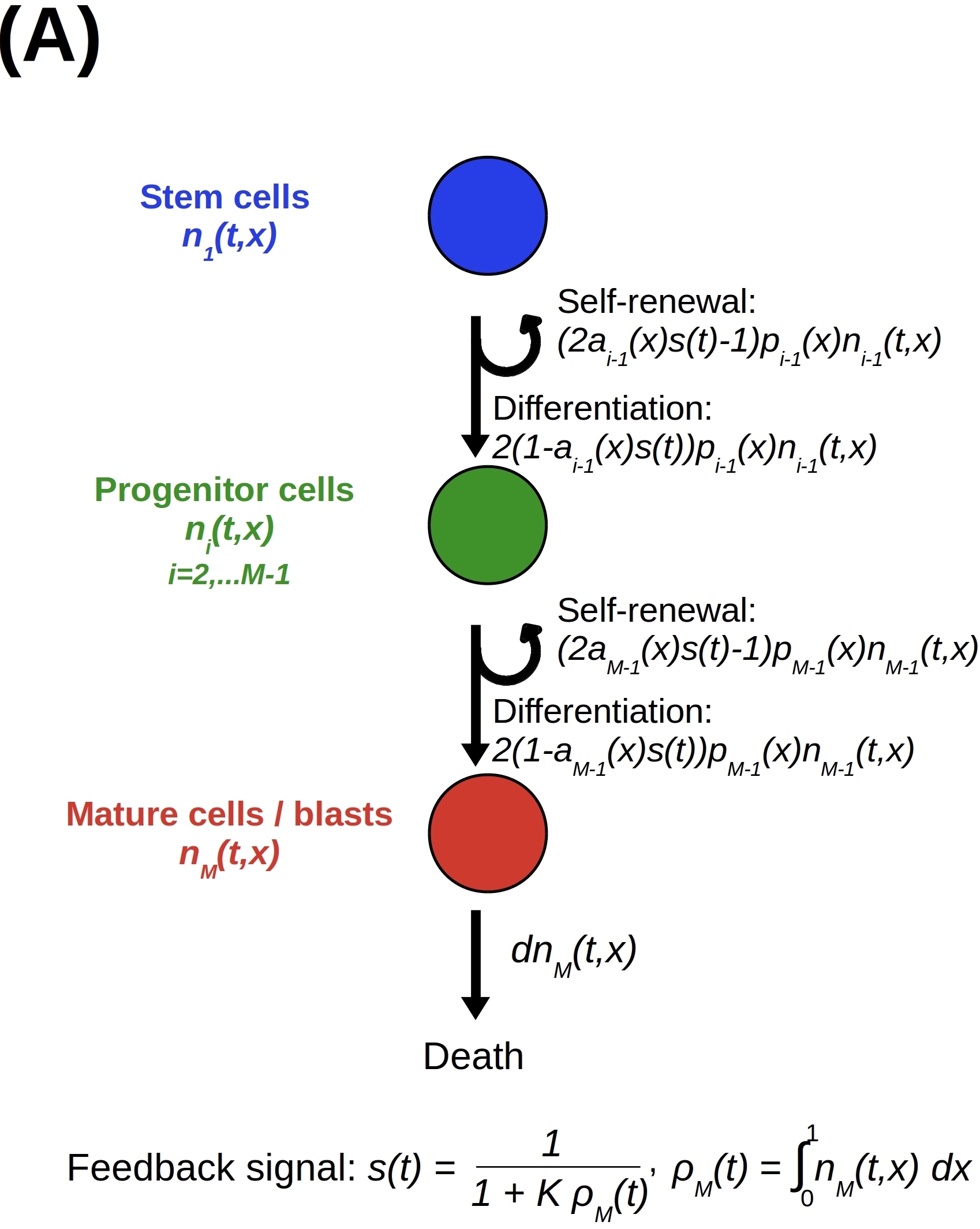}
\includegraphics[width=14cm]{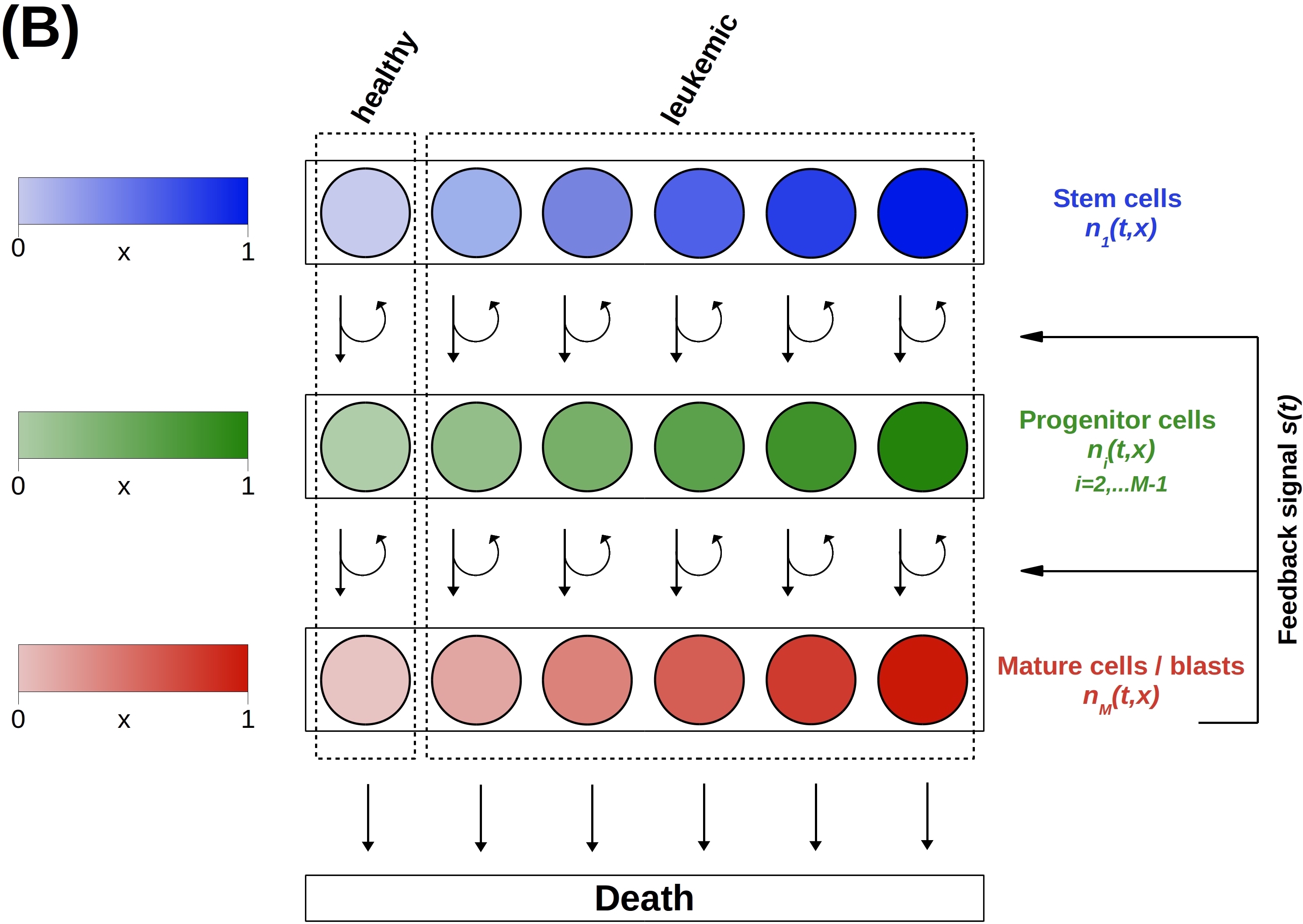}
\caption{{\bf Schematic overview of the model.} {\bf (\textsf{A})} Processes at the cellular level and their mathematical description  {\bf (\textsf{B})} Multi-compartmental continuously structured population model. For each maturation compartment $i$, different leukemic clones are characterised by different values of the continuous variable $x \in (0,1]$, whereas $x=0$ corresponds to healthy cells. All compartments are coupled through the feedback signal $s(t)$. Proliferation, self-renewal and differentiation of cells are modelled as schematised in panel (\textsf{A}).}
\label{fig:Model}
\end{figure}
At time $t \in [0,T]$, where $T>0$ is an arbitrary final time, the population densities of stem cells (compartment $i=1$), increasingly mature progenitor cells (compartments $i=2, \ldots, M-1$), and mature cells and leukemic blasts (compartment $i=M$) are represented by the functions $n_i(t,x) \geq 0$. At every time instant $t$, the total density of cells at the $i$-th maturation stage is defined as
\begin{equation}
\rho_i(t) := \int_0^1 n_i(t,x) \, {\rm d}x \quad \mbox{with} \quad i=1, \ldots, M.
\end{equation}

As illustrated by the scheme in Fig.~\ref{fig:Model}(\textsf{A}), mature cells and leukemic blasts do not divide and die at a constant rate $d>0$. Moreover, we use the functions $p_i(x)$ and $a_i(x)$ to model, respectively, the proliferation rate and the self-renewal fraction of cells of clone $x$ at the maturation stage $i=1,\ldots,M-1$. The flux to mitosis of cells of clone $x$ at the maturation stage $i=1,\ldots,M-1$ at time $t$ is given by $n_i(t,x) \, p_i(x)$. These $n_i(t,x) \, p_i(x)$ cells divide into $2 \, n_i(t,x) \, p_i(x)$ offspring cells. The fraction $a_i(x) \, s(t)$ of the offspring cells is at the maturation stage $i$ (self-renewal), while the fraction $1 - a_i(x) \, s(t)$ is at the maturation stage $i+1$ (differentiation). The factor $s(t)$ models the concentration of feedback signal that promotes the self-renewal of dividing cells and is absorbed by mature cells and leukemic blasts. In agreement with the biological findings presented in \citep{Kondo,Layton,Shinjo}, we let $s(t)$ be a monotonically decreasing function of $\rho_M(t)$. In particular, we use the definition $s(t) = \frac{1}{1 + K \rho_M(t)}$, where the parameter $K>0$ is related to the degradation rate of the feedback signal~\citep{MarciniakStiehl,StiehlBMT,StiehlMCM,StiehlMMNP}. Such a definition can be derived from an ODE for $s(t)$ using a quasi-stationary approximation~\citep{Getto,Mikelic}. 

Throughout the paper, we assume 
\begin{equation}
\label{ass.a}
\displaystyle{a_i \in {\rm W}^{1,\infty}\left([0,1] \right) \; \mbox{ with } \; a_i: [0,1] \rightarrow \left(\frac{1}{2}, 1\right) \; \mbox{ for all } \; i = 1, \ldots, M-1,}
\end{equation}
\begin{equation}
\label{ass.p}
\displaystyle{p_i \in {\rm W}^{1,\infty}\left([0,1] \right) \; \mbox{ with } \; p_i: [0,1] \rightarrow \left(0, 1\right) \; \mbox{ for all } \; i = 1, \ldots, M-1,}
\end{equation}
\begin{equation}
\label{ass.a1i}
\frac{1}{2} < a_i(x) < a_1(x) < 1 \; \mbox{ for all } \; x \in [0,1], \; \mbox{ for all } i = 2, \ldots, M-1.
\end{equation}
Assumptions~\eqref{ass.a}-\eqref{ass.a1i} rely on the following biological considerations. It has been shown that the self-renewal fraction of stem cells has to be larger than $\frac{1}{2}$ to allow for cell expansion~\citep{StiehlMCM,StiehlMMNP}. This justifies assumption~\eqref{ass.a} for $i=1$. Moreover, we justify assumptions~\eqref{ass.a} for $i=2, \ldots, M-1$ on the basis of biological evidence indicating that progenitor cells at different maturation stages are able to expand with little or no influx from the stem-cell compartment~\citep{Roelofs,Yamamoto}. Previous experimental and theoretical studies~\citep{Hope,Reya,StiehlMCM} have shown that stem cells have a higher self-renewal fraction than the progenitor cells to which they give rise. This justifies assumption~\eqref{ass.a1i}. Furthermore, the cellular proliferation rates are bounded from above due to the time required by genome replication. A reasonable upper bound is between 1 and 2 divisions per day~\citep{AlbertsCC}, from which one can deduce assumptions~\eqref{ass.p}. 

The evolution of the cell population density functions $n_i(t,x)$ is governed by the following system of coupled IDEs
\begin{equation}
\label{e.mod3o}
\left\{
\begin{array}{ll}
\displaystyle{\frac{\p}{\p t}n_1(t,x) =  \left(\frac{2 \, a_1(x)}{1 + K \rho_M(t)} - 1 \right) p_1(x) \, n_1(t,x),}
\\\\
\displaystyle{\frac{\p}{\p t}n_i(t,x) =  2 \left(1 - \frac{a_{i-1}(x)}{1 + K \rho_M(t)} \right) p_{i-1}(x) \, n_{i-1}(t,x)} 
\\\\
\displaystyle{\phantom{\frac{\p n_i}{\p t}(t,x) =} \; + \left(\frac{2 \, a_i(x)}{1 + K \rho_M(t)} - 1 \right) p_i(x) \, n_i(t,x), \quad  i=2, \ldots, M-1,}
\\\\
\displaystyle{\frac{\p}{\p t}n_M(t,x) = 2 \left(1 - \frac{a_{M-1}(x)}{1 + K \rho_M(t)} \right) p_{M-1}(x) \, n_{M-1}(t,x) - d \, n_M(t,x),}
\\\\
\displaystyle{\rho_{M}(t) = \int_0^1 n_{M}(t,x) \, {\rm d}x,}
\end{array}
\right.
\end{equation}
subject to the initial conditions below
\begin{equation}
\label{e.mod3IC}
n_{i}(0,x) = n_{i}^0(x) \in {\rm C}([0,1]), \;\; 0 < n_{i}^0 < \infty \;\; \mbox{on } [0,1] \;\; \mbox{ for } \; i=1, \ldots, M,
\end{equation}
which correspond to a biologically consistent scenario where numerous leukemic clones are present at the time of diagnosis. 

\section{Analysis of clonal selection}
We prove a general asymptotic result (\emph{vid.} Section \ref{subsec41}) that sheds light on the way in which the self-renewal fraction and the proliferation rate of cells at different maturation stages impact on clonal selection (\emph{vid.} Section \ref{subsec42}). 

\subsection{A general asymptotic result}
\label{subsec41}
Building on previous studies on the long-time behaviour of continuously structured population models \citep{barles2009concentration,chisholm2016effects,desvillettes2008selection,diekmann2005dynamics,lorz2011dirac,perthame2008dirac}, we introduce a small parameter $\e > 0$ and use the time scaling $t \mapsto \frac{t}{\e}$, so that considering the asymptotic regime $\e \rightarrow 0$ is equivalent to studying the behaviour of the solutions to the IDEs of the model over many cell generations. With this scaling, the Cauchy problem for the cell population density functions $n_{i}\left(\frac{t}{\e},x\right)=n_{i \e}(t,x)$ reads as
\begin{equation}
\label{e.mod3}
\left\{
\begin{array}{ll}
\displaystyle{\e \, \frac{\p}{\p t}n_{1 \e}(t,x) = P_{1}(\rho_{M \e}(t),x)\, n_{1 \e}(t,x),}
\\\\
\displaystyle{\e \, \frac{\p}{\p t}n_{i \e}(t,x) =  Q_{i-1}(\rho_{M \e}(t),x)  \, n_{i-1 \e}(t,x)} \\
\displaystyle{\phantom{\e \, \frac{\p n_{i \e}}{\p t}(t,x) =} \; + P_{i}(\rho_{M \e}(t),x) \, n_{i \e}(t,x), \quad  i=2, \ldots, M-1,}
\\
\displaystyle{\e \, \frac{\p}{\p t}n_{M\e}(t,x) = Q_{M-1}(\rho_{M \e}(t),x)  \, n_{M-1 \e}(t,x) - d \, n_{M \e}(t,x),}
\\\\
\displaystyle{\rho_{M \e}(t) = \int_0^1 n_{M \e}(t,x) \, {\rm d}x,}
\\\\
n_{i\e}(0,x) = n_{i}^0(x), \quad  i=1, \ldots, M,
\end{array}
\right.
\end{equation}
where
$$
P_{i}(\rho_{M \e},x) := \left(\frac{2 \, a_i(x)}{1 + K \rho_{M \e}} - 1 \right) \, p_i(x) \; \mbox{ for } \; i=1, \ldots, M-1 
$$
and
$$
Q_{i}(\rho_{M \e},x) := 2 \left(1 - \frac{a_i(x)}{1 + K \rho_{M \e}} \right) \, p_i(x)  \; \mbox{ for } \; i=1, \ldots, M-1.  
$$
Throughout this section we use the notation
$$
\rho_{i \e}(t) := \int_0^1 n_{i \e}(t,x) \, {\rm d}x \quad \mbox{with} \quad i=1, \ldots, M
$$
and
\begin{equation}\label{eq:Re_C3}
R_{i \e}(t,x) := \int_0^t P_{i}(\rho_{M \e}(s),x) \, {\rm d}s \; \mbox{ for } \; i=1,\ldots,M-1.
\end{equation}

A general result on the asymptotic behaviour of the cell population density functions $n_{i \e}(t,x)$ for $\e \rightarrow 0$ (\emph{i.e.} in the limit of many cell generations) is established by Theorem~\ref{th:case3}.
\begin{theorem}[A general asymptotic result]
\label{th:case3}
Under assumptions \eqref{ass.a}-\eqref{ass.a1i} and \eqref{e.mod3IC}, the solutions to the Cauchy problem \eqref{e.mod3} are such that, up to extraction of subsequences,
\begin{equation}\label{eq:nle_C3}
n_{i\e} \xrightharpoonup[\varepsilon  \rightarrow 0]{} n_{i}  \; \mbox{ on } \; w^*-{\rm L}^{\infty}\left((0,T), {\rm M}^1([0,1])\right) \mbox{ for } i=1, \ldots, M
\end{equation}
and
\begin{equation}\label{eq:Rleunif_C3}
R_{i \e} \xrightarrow[\varepsilon  \rightarrow 0]{} R_i \; \mbox{uniformly in } \; [0,T] \times [0,1] \; \mbox{ for } \; i=1,\ldots,M-1,
\end{equation}
where
\begin{equation}\label{eq:R_C3}
R_i(t,x) := \int_0^t P_{i}(\rho_{M}(s),x) \, {\rm d}s \quad \text{and} \quad R_i \in {\rm W}^{1,\infty}\left((0,T) \times [0,1]\right).
\end{equation}
Moreover, the limits $R_i$ are such that 
\begin{equation}
\label{eq:step3a_C3th1}
\max_{x \in [0,1]} R_1(t,x)=0 \; \mbox{ for any } t \in [0,T]
\end{equation}
and
\begin{equation}
\label{eq:step3a_C3th2}
\max_{x \in [0,1]} R_i(t,x)<0 \; \mbox{ for any } t \in [0,T],  \; \mbox{ for } \; i=2, \ldots, M-1.
\end{equation}
Finally, the limits $n_i$ are such that for a.e. $t \in [0,T]$ 
\begin{equation}\label{eq:suppn1_C3}
\supp(n_i(t, \cdot)) = \argmax{x \in [0,1]} a_1(x) \mbox{ for } i=1, \ldots, M.
\end{equation}
\end{theorem}

\begin{proof} We divide the proof of Theorem \ref{th:case3} into five parts.
\\\\
{\bf Part 1: Non-negativity of $n_{i\e}(t,x)$ and continuity of $\rho_{i \e}(t)$}.
For all $\e >0$, standard arguments based on the Banach fixed point theorem allow one to prove that, under assumptions \eqref{ass.a}-\eqref{ass.a1i} and \eqref{e.mod3IC}, the Cauchy problem~\eqref{e.mod3} admits a solution of non-negative components $n_{i \e} \in {\rm C}([0,\infty),{\rm L}^1([0,1]))$. 
\\\\
{\bf Part 2: Uniform upper bounds for $\rho_{i \e}(t)$}. The following uniform upper bounds hold
 \begin{equation}
\label{eq:UBfinal}
\rho_{i \e}(t) < \overline{\rho}_{i} \; \mbox{ for } \; i=1, \ldots, M, \quad \mbox{ for all } \; t \geq 0 \; \mbox{ and for any } \; \e >0
\end{equation}
with $0<\overline{\rho}_{i}<\infty$, the proof of which is provided in Appendix~B. 
\\\\
{\bf Part 3: Proof of \eqref{eq:nle_C3}}. The upper bounds~\eqref{eq:UBfinal} allow us to use the Banach-Alaoglu theorem to conclude that, up to extraction of subsequences, the asymptotic result~\eqref{eq:nle_C3} is verified.
\\\\
{\bf Part 4: Proof of \eqref{eq:Rleunif_C3} and \eqref{eq:R_C3}.} 
The asymptotic result~\eqref{eq:nle_C3} on $n_{M \, \e}$ allows us to conclude that there exists a subsequence of $R_{i \e}$, that we denote again as $R_{i \e}$, such that
\begin{equation}\label{eq:Rle_C3}
\mbox{ for all } (t,x) \in [0,\infty) \times [0,1] \quad \text{we have} \quad R_{i \e}(t,x) \xrightarrow[\varepsilon  \rightarrow 0]{} R_{i}(t,x),
\end{equation}
for all $i=1,\ldots,M-1$, where $R_{i}(t,x)$ is defined according to~\eqref{eq:R_C3}. Furthermore, assumptions \eqref{ass.a} and \eqref{ass.p} ensure that, for any $\e>0$, the function $R_{i \e}$ and its first derivatives with respect to $t$ and $x$ are bounded in ${\rm L}^{\infty}\left((0,T) \times [0,1]\right)$. Therefore, $R_{i \e}$ belongs to ${\rm W}^{1,\infty}\left((0,T) \times [0,1]\right)$ for any $\e>0$ and, using the fact that ${\rm W}^{1,\infty}((0,T) \times [0,1])$ is compactly embedded in ${\rm C}([0,T] \times [0,1])$, we conclude that the uniform convergence result~\eqref{eq:Rleunif_C3} is verified.\\

\noindent {\bf Part 5: Proof of \eqref{eq:step3a_C3th1}-\eqref{eq:suppn1_C3}}. We prove the results~\eqref{eq:step3a_C3th1}-\eqref{eq:suppn1_C3} in four steps.
\\\\
{\bf Part 5 -- Step 1.} We prove that 
\begin{equation}
\label{eq:step1_C3}
R_{1}(t,x) \leq 0 \; \mbox{ for all } \; (t,x) \in (0,T) \times [0,1]
\end{equation}
and
\begin{equation}
\label{eq:step1_C3bis}
R_{i}(t,x) < 0 \; \mbox{ for all } \; (t,x) \in (0,T) \times [0,1], \; \mbox{ for all } i =2, \ldots, M-1.
\end{equation}
By contradiction, assume that there exists $(\hat{t},\hat{x}) \in [0,T] \times (0,1)$ such that $R_1(\hat{t},\hat{x}) > 0$. The convergence result~\eqref{eq:Rleunif_C3} for $i=1$ implies that $R_{1 \e}(t,x) \geq \sigma$ for some $\sigma > 0$, as long as $|t - \hat{t}| \leq \sigma$, $|x - \hat{x}| \leq \sigma$ and $\sigma \geq \e > 0$. Since $n_1^0>0$ on $[0,1]$ we conclude that
$$
\int_0^1 n_{1 \e}(t,x) \, {\rm d}x = \int_0^1 n_1^0(x) \, e^{\frac{R_{1 \e}(t,x)}{\e}} \, {\rm d}x \geq e^{\frac{\sigma}{\e}} \, \int_{\hat{x}-\sigma}^{\hat{x}+\sigma} n_1^0(x) \, {\rm d}x \xrightarrow[\varepsilon  \rightarrow 0]{} \infty 
$$
for all $t \in [\hat{t}-\sigma,\hat{t}+\sigma]$. This contradicts the upper bound~\eqref{eq:UBfinal} for $\rho_{1 \e}(t)$, thus proving~\eqref{eq:step1_C3}. Moreover, such a result on $R_1$ ensures that
$$
\int_0^t P_{1}(\rho_{M}(s),x) \, {\rm d}s  = p_1(x) \, \int_0^t \left(\frac{2 \, a_1(x)}{1 + K \rho_{M}(s)} - 1 \right) \, {\rm d}s \leq 0
$$
for all $(t,x) \in (0,T) \times [0,1]$ and, since $p_1 > 0$ on $[0,1]$ [\emph{cf.} assumption \eqref{ass.p}], the above inequality implies that
$$ 
\int_0^t \left(\frac{2 \, a_1(x)}{1 + K \rho_{M}(s)} - 1 \right) \, {\rm d}s \leq 0 \; \mbox{ for all } \; (t,x) \in (0,T) \times [0,1].
$$
Since $a_1(x) > a_i(x)$ for all $x \in [0,1]$ and for all $i=2,\ldots, M-1$ [\emph{cf.} assumption \eqref{ass.a1i}], the latter inequality allows one to conclude that for all $i =2, \ldots, M-1$
$$ 
\int_0^t \left(\frac{2 \, a_i(x)}{1 + K \rho_{M}(s)} - 1 \right) \, {\rm d}s < 0 \; \mbox{ for all } \; (t,x) \in (0,T) \times [0,1]
$$
and using the fact that $p_{i}>0$ on $[0,1]$ [\emph{cf.} assumption \eqref{ass.p}] one obtains
$$ 
p_i(x) \, \int_0^t \left(\frac{2 \, a_i(x)}{1 + K \rho_{M}(s)} - 1 \right) \, {\rm d}s < 0 \; \mbox{ for all } \; (t,x) \in (0,T) \times [0,1]
$$
for all $i =2, \ldots, M-1$. This concludes the proof of~\eqref{eq:step1_C3bis}, which implies that~\eqref{eq:step3a_C3th2} is verified.
\\\\
{\bf Part 5 -- Step 2.} We prove that
\begin{equation}
\label{eq:n=0ae_C3}
\mbox{if } \; R_1(t,\cdot) < 0 \; \mbox{ on } \; [0,1] \; \mbox{ then } \; n_i(t,\cdot) = 0 \; \mbox{ a.e. on } [0,1] 
\end{equation}
for all $i =1, \ldots, M$. Throughout this step we consider $(\hat{t},\hat{x}) \in (0,T) \times (0,1)$ such that $R_1(\hat{t},\hat{x}) < 0$.\\

\noindent \emph{Proof of~\eqref{eq:n=0ae_C3} for $i=1$.}  The uniform convergence result~\eqref{eq:Rleunif_C3} for $i=1$ ensures that there exists some $\sigma > 0$ such that $R_{1 \e}(t,x) \leq - \sigma$ for $|t - \hat{t}| \leq \sigma$, $|x - \hat{x}| \leq \sigma$ and $\sigma \geq \e > 0$. This allows us to conclude that
\begin{eqnarray*}
 \lim_{\e \rightarrow 0} \int_{\hat{t}-\sigma}^{\hat{t}+\sigma} \int_{\hat{x}-\sigma}^{\hat{x}+\sigma} n_{1 \e}(t,x) \, {\rm d}x \, {\rm d}t & = & \lim_{\e \rightarrow 0} \int_{\hat{t}-\sigma}^{\hat{t}+\sigma} \int_{\hat{x}-\sigma}^{\hat{x}+\sigma} n_1^0(x) e^{\frac{R_{1 \e}(t,x)}{\e}} \, {\rm d}x \, {\rm d}t
\nonumber \\
& \leq & 2 \, \sigma \lim_{\e \rightarrow 0} e^{-\frac{\sigma}{\e}}  \int_{\hat{x}-\sigma}^{\hat{x}+\sigma} n_1^0(x) \, {\rm d}x,
\end{eqnarray*}
that is,
\begin{equation}
\label{e.npn1a}
 \lim_{\e \rightarrow 0} \int_{\hat{t}-\sigma}^{\hat{t}+\sigma} \int_{\hat{x}-\sigma}^{\hat{x}+\sigma} n_{1 \e}(t,x) \, {\rm d}x \, {\rm d}t = 0.
\end{equation}
The weak convergence result~\eqref{eq:nle_C3} for $n_{1 \e}(t,x)$ ensures that
\begin{equation}
\label{e.npn1b}
 \int_{\hat{t}-\sigma}^{\hat{t}+\sigma} \int_0^1 \varphi(x) n_1(t,x) \,{\rm d}x\, {\rm d}t = \lim_{\e \rightarrow 0}  \int_{\hat{t}-\sigma}^{\hat{t}+\sigma} \int_0^1 \varphi(x) n_{1 \e}(t,x) \,{\rm d}x\, {\rm d}t
\end{equation}
for every test function $\varphi : [0,1] \rightarrow \mathbb{R}$. Therefore, choosing a smooth test function $\varphi$ that satisfies the following conditions
\begin{equation}
\label{a.testfunc}
\displaystyle{\mathbf{1}_{[\hat{x}-\sigma/2,\hat{x}+\sigma/2]} \leq \varphi \leq  \mathbf{1}_{[\hat{x}-\sigma,\hat{x}+\sigma]}},
\end{equation}
where $\mathbf{1}$ denotes the indicator function, and using the fact that $n_{1 \e}$ is non-negative we find
\begin{equation}
\label{e.nintineq1}
 \int_{\hat{t}-\sigma}^{\hat{t}+\sigma} \int_0^1 \varphi(x) n_1(t,x) \,{\rm d}x\, {\rm d}t \leq \lim_{\e \rightarrow 0}  \int_{\hat{t}-\sigma}^{\hat{t}+\sigma}  \int_{\hat{x}-\sigma}^{\hat{x}+\sigma} n_{1 \e}(t,x) \, {\rm d}x \, {\rm d}t.
\end{equation}
Substituting~\eqref{e.npn1a} into the latter integral inequality we conclude that 
\begin{equation}
\label{e.npn1bsup}
\int_{\hat{t}-\sigma}^{\hat{t}+\sigma} \int_0^1 \varphi(x) n_1(t,x) \,{\rm d}x\, {\rm d}t = 0
\end{equation}
for every smooth test function that satisfies~\eqref{a.testfunc}. Hence, the result~ \eqref{eq:n=0ae_C3} for $i=1$ is verified.\\

\noindent \emph{Proof of~\eqref{eq:n=0ae_C3} for $i=2$.} Multiplying by a test function $\varphi : [0,1] \rightarrow \mathbb{R}$ both sides of the IDE~\eqref{e.mod3} for $n_{2 \e}$ and integrating over the set $[\hat{t}-\sigma,\hat{t}+\sigma] \times [0,1]$ we find
\begin{eqnarray*}
&& \e \left[\int_0^1 \varphi(x) \, n_{2 \e}(\hat{t}+\sigma,x) \, {\rm d}x - \int_0^1 \varphi(x) \, n_{2 \e}(\hat{t}-\sigma,x) \, {\rm d}x \right] \\
&&  \phantom{aaaa} = \int_{\hat{t}-\sigma}^{\hat{t}+\sigma} \int_0^1 Q_{1}(\rho_{M \, \e}(t),x) \, \varphi(x) \, n_{1 \e}(t,x) \, {\rm d}x \, {\rm d}t \\
&&  \phantom{aaaaaaa} + \int_{\hat{t}-\sigma}^{\hat{t}+\sigma} \int_0^1 P_{2}(\rho_{M \, \e}(t),x) \, \varphi(x) \, n_{2 \e}(t,x) \, {\rm d}x \, {\rm d}t.
\end{eqnarray*}
Since the uniform upper bound~\eqref{eq:UBfinal} for $\rho_{2 \e}$ ensures that
$$
\lim_{\varepsilon \rightarrow 0} \e \left[\int_0^1 \varphi(x) \, n_{2 \e}(\hat{t}+\sigma,x) \, {\rm d}x - \int_0^1 \varphi(x) \, n_{2 \e}(\hat{t}-\sigma,x) \, {\rm d}x \right]  = 0,
$$
we conclude that
\begin{eqnarray*}
&& \lim_{\e \rightarrow 0} \int_{\hat{t}-\sigma}^{\hat{t}+\sigma} \int_0^1 Q_{1}(\rho_{M \, \e}(t),x) \, \varphi(x) \, n_{1 \e}(t,x) \, {\rm d}x \, {\rm d}t  \\
&&  \phantom{aaaa} + \lim_{\e \rightarrow 0} \int_{\hat{t}-\sigma}^{\hat{t}+\sigma} \int_0^1 P_{2}(\rho_{M \, \e}(t),x) \, \varphi(x) \, n_{2 \e}(t,x) \, {\rm d}x \, {\rm d}t = 0.
\end{eqnarray*}
Choosing a smooth test function that satisfies~\eqref{a.testfunc} and using~\eqref{e.npn1b} and~\eqref{e.npn1bsup} along with the fact that $Q_{1}(\rho_{M \, \e}(t),x) > 0$ on $[0,T] \times [0,1]$ yields 
\begin{equation}
\label{e.P2en2e}
\lim_{\e \rightarrow 0} \int_{\hat{t}-\sigma}^{\hat{t}+\sigma} \int_{0}^{1} P_{2}(\rho_{M \, \e}(t),x) \, \varphi(x) \, n_{2 \e}(t,x) \, {\rm d}x \, {\rm d}t = 0.
\end{equation}
Since $P_{2}(\rho_{M \, \e}(t),x)<0$ on $[0,1]$ for a.e. $t \in (0,T)$ when $\e \to 0$, the result given by \eqref{e.P2en2e} allows us to conclude that
\begin{equation}
\label{e.intn20}
\lim_{\e \rightarrow 0} \int_{\hat{t}-\sigma}^{\hat{t}+\sigma} \int_{0}^{1} \varphi(x) \, n_{2 \e}(t,x) \, {\rm d}x \, {\rm d}t = 0,
\end{equation}
for every smooth test function that satisfies~\eqref{a.testfunc}. The weak convergence result~\eqref{eq:nle_C3} for $n_{2 \e}(t,x)$ ensures that
\begin{equation}
\label{e.npn1bbb}
\int_{\hat{t}-\sigma}^{\hat{t}+\sigma} \int_0^1 \varphi(x) \, n_2(t,x) \, {\rm d}x \, {\rm d}t = \lim_{\e \rightarrow 0} \int_{\hat{t}-\sigma}^{\hat{t}+\sigma} \int_0^1 \varphi(x) \, n_{2 \e}(t,x) \, {\rm d}x \, {\rm d}t.
\end{equation}
Hence, using~\eqref{e.intn20} and~\eqref{e.npn1bbb} we conclude that
$$
\int_{\hat{t}-\sigma}^{\hat{t}+\sigma}\int_0^1 \varphi(x) \, n_2(t,x) \,{\rm d}x\, {\rm d}t = 0
$$
for every smooth test function that satisfies~\eqref{a.testfunc}. Therefore, the result~\eqref{eq:n=0ae_C3} for $i=2$ is verified.\\

\noindent \emph{Proof of~\eqref{eq:n=0ae_C3} for all $ i=3,\ldots,M-1$.} Using a bootstrap argument based on the method of proof that we have used for the case $i=2$, one can prove that for all $i=3,\ldots,M-1$
\begin{equation}
\label{e.intnig0}
\lim_{\e \rightarrow 0} \int_{\hat{t}-\sigma}^{\hat{t}+\sigma} \int_{0}^{1} \varphi(x) \, n_{i \e}(t,x) \, {\rm d}x \, {\rm d}t  = \int_{\hat{t}-\sigma}^{\hat{t}+\sigma} \int_{0}^{1} \varphi(x) \, n_{i}(t,x) \, {\rm d}x \, {\rm d}t  = 0 
\end{equation}
for every smooth test function that satisfies~\eqref{a.testfunc}. Hence, the results~\eqref{eq:n=0ae_C3} for $i~=~3,\ldots,M-1$ are verified. \\

\noindent \emph{Proof of~\eqref{eq:n=0ae_C3} for $i=M$.} Multiplying by a test function $\varphi : [0,1] \rightarrow \mathbb{R}$ both sides of the IDE~\eqref{e.mod3} for $n_{M \e}$ and integrating over the set $[\hat{t}-\sigma,\hat{t}+\sigma] \times [0,1]$ we obtain
\begin{eqnarray*}
&& \e \left[\int_0^1 \varphi(x) \, n_{M \e}(\hat{t}+\sigma,x) \, {\rm d}x - \int_0^1 \varphi(x) \, n_{M \e}(\hat{t}-\sigma,x) \, {\rm d}x \right] \\
&& \phantom{aaaa} = \int_{\hat{t}-\sigma}^{\hat{t}+\sigma} \int_0^1 Q_{M-1}(\rho_{M \, \e}(t),x) \, \varphi(x) \, n_{M-1 \e}(t,x) \, {\rm d}x \, {\rm d}t \\
&& \phantom{aaaaaaaa} - d \, \int_{\hat{t}-\sigma}^{\hat{t}+\sigma} \int_0^1 \varphi(x) \, n_{M \e}(t,x) \, {\rm d}x \, {\rm d}t.
\end{eqnarray*}
Since the uniform upper bound~\eqref{eq:UBfinal} for $\rho_{M \e}$ ensures that
$$
\lim_{\varepsilon \rightarrow 0} \e \left[\int_0^1 \varphi(x) \, n_{M \e}(\hat{t}+\sigma,x) \, {\rm d}x - \int_0^1 \varphi(x) \, n_{M \e}(\hat{t}-\sigma,x) \, {\rm d}x \right] = 0,
$$
we conclude that
\begin{eqnarray*}
&& \lim_{\e \rightarrow 0} \int_{\hat{t}-\sigma}^{\hat{t}+\sigma} \int_0^1 Q_{M-1}(\rho_{M \, \e}(t),x) \, \varphi(x) \, n_{M-1 \e}(t,x) \, {\rm d}x \, {\rm d}t \\
&& \phantom{aaaa}  - d \, \lim_{\e \rightarrow 0} \int_{\hat{t}-\sigma}^{\hat{t}+\sigma} \int_0^1 \varphi(x) \, n_{M \e}(t,x) \, {\rm d}x \, {\rm d}t = 0.
\end{eqnarray*}
Choosing a smooth test function that satisfies~\eqref{a.testfunc} and using the result~\eqref{e.intnig0} for $n_{M-1 \e}(t,x)$ along with the fact that $Q_{M-1}(\rho_{M \, \e}(t),x) > 0$ on $[0,T] \times [0,1]$ gives
\begin{equation}
\label{e.intn30}
\lim_{\e \rightarrow 0} \int_{\hat{t}-\sigma}^{\hat{t}+\sigma} \int_{0}^{1}  \varphi(x) \, n_{M \e}(t,x) \, {\rm d}x \, {\rm d}t = 0.
\end{equation}
The weak convergence result \eqref{eq:nle_C3} for $n_{M \e}(t,x)$ ensures that
\begin{equation}
\label{e.npn1c}
\int_{\hat{t}-\sigma}^{\hat{t}+\sigma} \int_0^1 \varphi(x) \, n_M(t,x) \,{\rm d}x \, {\rm d}t = \lim_{\e \rightarrow 0} \int_{\hat{t}-\sigma}^{\hat{t}+\sigma} \int_0^1 \varphi(x) \, n_{M \e}(t,x) \,{\rm d}x \, {\rm d}t.
\end{equation}
Hence, using~\eqref{e.intn30} and~\eqref{e.npn1c} we conclude that
$$
\int_{\hat{t}-\sigma}^{\hat{t}+\sigma} \int_0^1 \varphi(x) \, n_M(t,x) \, {\rm d}x \, {\rm d}t = 0
$$
for every smooth test function that satisfies~\eqref{a.testfunc}. Therefore, the result~\eqref{eq:n=0ae_C3} for $i=M$ is verified.
\\\\
{\bf Part 5 -- Step 3.} We prove that
\begin{equation}
\label{eq:step3a_C3}
\max_{x \in [0,1]} R_1(t,x)=0 \; \mbox{ for any } t \in [0,T],
\end{equation}
\begin{equation}
\label{eq:step3c_C3}
\rho_1(t) > 0 \; \mbox{ a.e. on } \; [0,T],
\end{equation}
and
\begin{equation}
\label{eq:step3b_C3}
\argmax{x \in [0,1]} R_1(t,x) = \argmax{x \in [0,1]} a_1(x) \; \mbox{ for any } t \in [0,T].
\end{equation}
We begin by noting that, since $p_1 > 0$ on $[0,1]$, if 
$$
\displaystyle{\max_{x \in [0,1]} R_1(t,x) = 0}
$$ 
for some $t \in [0,T]$ then [\emph{vid.} the definition~\eqref{eq:R_C3} of the function $R_1$]
$$
\max_{x \in [0,1]} \int_0^t \left(\frac{2 \, a_{1}(x)}{1 + K \rho_{M}(s)} - 1 \right) \, {\rm d}s = 0
$$
and, therefore, 
\begin{equation}
\label{eq:hatx_C3}
\argmax{x \in [0,1]} R_1(t,x) = \argmax{x \in [0,1]} a_1(x).
\end{equation}
Hence, if~\eqref{eq:step3a_C3} holds true then~\eqref{eq:step3b_C3} is verified. 

To prove~\eqref{eq:step3a_C3} and \eqref{eq:step3c_C3} we proceed as follows. Assume by contradiction that there exist $\hat{t} \in [0,T)$ and $\sigma > 0$ with  $\hat{t}+\sigma \leq T$ such that
\begin{equation}
\label{eq:assmaxR_C3}
\max_{x \in [0,1]} R_1(t,x) = 0 \quad \forall t \in [0,\hat{t}] \; \mbox{ and } \; \rho_1(t) > 0 \; \mbox{ a.e. on } \; [0,\hat{t}]
\end{equation}
whereas 
\begin{equation}
\label{eq:assmaxR2_C3}
\max_{x \in [0,1]} R_1(t,x) < 0 \quad \forall t \in (\hat{t},\hat{t}+\sigma).
\end{equation}
Under assumptions \eqref{eq:assmaxR_C3} and \eqref{eq:assmaxR2_C3}, the result~\eqref{eq:n=0ae_C3} on $n_M(t,x)$ allows one to conclude that
\begin{equation}\label{eq:implrho2_C3}
\rho_M(t) = 0 \; \mbox{ for a.e.} \; t \in (\hat{t},\hat{t}+\sigma).
\end{equation}
We take $\displaystyle{\hat{x} \in \argmax{x \in [0,1]} R_1(\hat{t},x)}$. Using~\eqref{eq:implrho2_C3} we find that under assumptions \eqref{eq:assmaxR_C3} and \eqref{eq:assmaxR2_C3}
$$
R_1(\hat{t}+\sigma,\hat{x}) = \int_{0}^{\hat{t}+\sigma} P_1(\rho_M(t),\hat{x}) \,{\rm d}t = R_1(\hat{t},\hat{x}) + \int_{\hat{t}}^{\hat{t}+\sigma} P_1(0,\hat{x}) \,{\rm d}t, 
$$
that is,
$$
R_1(\hat{t}+\sigma,\hat{x}) = \int_{\hat{t}}^{\hat{t}+\sigma} p_1(\hat{x}) \, \Big[2 \, a_1(\hat{x}) - 1 \Big] \,{\rm d}t.
$$
Due to assumption \eqref{ass.a} this yields $R_1(\hat{t}+\sigma,\hat{x}) > 0$, which contradicts the result~\eqref{eq:step1_C3} on $R_1$. Hence, 
$$
\max_{x \in [0,1]} R_1(t,x)=0 \; \text{ for any } \; t \in (\hat{t},\hat{t}+\sigma) \; \text{ and } \; \rho_1(t) > 0 \; \text{ a.e on } (\hat{t},\hat{t}+\sigma)
$$
as well. Therefore, the results~\eqref{eq:step3a_C3} and \eqref{eq:step3c_C3} are verified. Hence, both the result~\eqref{eq:step3a_C3th1} and, under the assumption~\eqref{e.mod3IC} on $n_1^0$, the result~\eqref{eq:suppn1_C3} on the limit $n_1$ hold.
\\\\
{\bf Part 5 -- Step 4.} To prove the result~\eqref{eq:suppn1_C3} on the limit $n_2$ we proceed as follows. The weak convergence result~\eqref{eq:nle_C3} for $n_{1 \e}$ along with the result~\eqref{eq:suppn1_C3} on the limit $n_1$ and the fact that $Q_{1}(\rho_{M \, \e}(t),x) > 0$ on $[0,T] \times [0,1]$ ensures that for all $\tau \in [0,T)$ and for all $\sigma > 0$ with $\tau + \sigma \leq T$ we have
\begin{equation}
\label{new1}
 \lim_{\e \rightarrow 0} \int_{\tau}^{\tau+\sigma} \int_{0}^{1} Q_{1}(\rho_{M \, \e}(t),x) \, n_{1 \e}(t,x) \, {\rm d}x \, {\rm d}t > 0.
\end{equation}
Moreover, integrating over the set $[\tau,\tau+\sigma] \times [0,1]$ both sides of the IDE~\eqref{e.mod3} for $n_{2 \e}$ we find
\begin{eqnarray*}
&& \e \left[\int_0^1 n_{2 \e}(\tau+\sigma,x) \, {\rm d}x - \int_0^1 n_{2 \e}(\tau,x) \, {\rm d}x \right] \\
&&  \phantom{aaaa} = \int_{\tau}^{\tau+\sigma} \int_0^1 Q_{1}(\rho_{M \, \e}(t),x) \, n_{1 \e}(t,x) \, {\rm d}x \, {\rm d}t \\
&&  \phantom{aaaaaaa} + \int_{\tau}^{\tau+\sigma} \int_0^1 P_{2}(\rho_{M \, \e}(t),x) \, n_{2 \e}(t,x) \, {\rm d}x \, {\rm d}t
\end{eqnarray*}
and, since the upper bound~\eqref{eq:UBfinal} for $\rho_{2 \e}$ ensures that
$$
\lim_{\varepsilon \rightarrow 0} \e \left[\int_0^1 n_{2 \e}(\tau+\sigma,x) \, {\rm d}x - \int_0^1 n_{2 \e}(\tau,x) \, {\rm d}x \right]  = 0,
$$
we conclude that
\begin{eqnarray*}
&& \lim_{\e \rightarrow 0} \int_{\tau}^{\tau+\sigma} \int_0^1 Q_{1}(\rho_{M \, \e}(t),x) \, n_{1 \e}(t,x) \, {\rm d}x \, {\rm d}t  \\
&&  \phantom{aaaa} + \lim_{\e \rightarrow 0} \int_{\tau}^{\tau+\sigma} \int_0^1 P_{2}(\rho_{M \, \e}(t),x) \, n_{2 \e}(t,x) \, {\rm d}x \, {\rm d}t = 0.
\end{eqnarray*}
This along with~\eqref{new1} implies that
\begin{equation}
\label{e.P2en2enew}
\lim_{\e \rightarrow 0} \int_{\tau}^{\tau+\sigma} \int_{0}^{1} P_{2}(\rho_{M \, \e}(t),x) \, n_{2 \e}(t,x) \, {\rm d}x \, {\rm d}t < 0.
\end{equation}
Coherently with the fact that $P_{2}(\rho_{M \, \e}(t),x)<0$ on $[0,1]$ for a.e. $t \in (0,T)$ when $\e \to 0$, the result given by~\eqref{e.P2en2enew} yields
\begin{equation}
\label{e.intn20new}
\lim_{\e \rightarrow 0} \int_{\tau}^{\tau+\sigma} \int_{0}^{1} n_{2 \e}(t,x) \, {\rm d}x \, {\rm d}t > 0.
\end{equation}
This along with the calculations carried out in Part 5 -- Step 2 allow us to conclude that the result~\eqref{eq:suppn1_C3} on the limit $n_2$ is verified. The results~\eqref{eq:suppn1_C3} for the limits $n_i$ with $i=3,\ldots,M$ can be proved in a similar way through a bootstrap argument. 
\begin{flushright}
\hfill $\Box$
\end{flushright}
\end{proof}

The general asymptotic result established by Theorem \ref{th:case3} put on a rigorous mathematical basis the idea that clonal selection is driven by the self-renewal fraction of leukemic stem cells, as exemplified by Corollaries \ref{cor3}-\ref{cor4} given in Section~\ref{subsec42}. 

\subsection{Biological implications of Theorem \ref{th:case3}}
\label{subsec42}
Building on the ideas presented in previous studies \citep{StiehlBaranHoMarciniakCR,StiehlMCM,StiehlMMNP}, here we focus on the case where there are $M=3$ possible maturation stages, that is, stem cells ($i=1$), progenitor cells ($i=2$) and mature cells/leukemic blasts ($i=3$). 
 
In this case, Corollary \ref{cor3} of Theorem \ref{th:case3} shows that if
\begin{equation}
\label{ass.acor3}
\argmax{x \in [0,1]} a_1(x) = \big\{\overline{x}\big\}
\end{equation}
then in the limit $\e \rightarrow 0$ the population density functions of stem cells $n_{1 \e}(t,x)$, progenitor cells $n_{2 \e}(t,x)$ and mature cells/leukemic blasts $n_{3 \e}(t,x)$ become concentrated as Dirac masses centred at the point $\overline{x}$. Analogously, Corollary~\ref{cor5} of Theorem \ref{th:case3} shows that if
\begin{equation}
\label{ass.acor5}
\argmax{x \in [0,1]} a_1(x) = \{\overline{x}_1, \dots, \overline{x}_N\},
\end{equation}
then when $\e \rightarrow 0$ the cell population density functions $n_{1 \e}(t,x)$, $n_{2 \e}(t,x)$ and $n_{3 \e}(t,x)$ become concentrated as weighted sums of Dirac masses centred at the points $\{\overline{x}_1, \dots, \overline{x}_N\}$. In both cases, the centres of the Dirac masses do not depend on the functions $p_1(x)$, $a_2(x)$ and $p_2(x)$. 

From a biological point of view, the centres of the Dirac masses can be understood as the leukemic clones that are selected for in the limit of many cell generations. Therefore, the asymptotic results of Corollary \ref{cor3} and Corollary \ref{cor5} formalise the idea that clonal selection is controlled by the self-renewal fraction of leukemic stem cells, as the leukemic clone(s) with the highest stem cell self-renewal fraction are ultimately selected. 

The results of Corollary~\ref{cor3} and Corollary~\ref{cor5} are complemented by Corollary~\ref{cor4} of Theorem \ref{th:case3}, which shows that if the function $a_1(x)$ is constant, \emph{i.e.} if 
\begin{equation}
\label{ass.acor4ap}
a_1(x) =  A_1 \in \mathbb{R}_+ \; \mbox{ with } \; \frac{1}{2} < a_2(x) < A_1 <1 \; \mbox{ for all } \; x \in [0,1],
\end{equation}
then in the asymptotic regime $\e \rightarrow 0$ the cell population density functions $n_{1 \e}(t,x)$, $n_{2 \e}(t,x)$ and $n_{3 \e}(t,x)$ do not become concentrated as Dirac masses. Biologically, this indicates that if the stem cell self-renewal fraction is the same for all leukemic clones, then clonal selection will not occur and no specific clones will be selected. 

\begin{cor}[Selection of one single clone] 
\label{cor3}
Assume $M=3$ and let the assumptions of Theorem \ref{th:case3} along with the additional assumption~\eqref{ass.acor3} hold. Then, the measures $n_1$, $n_2$ and $n_3$ given by Theorem \ref{th:case3} are such that
\begin{equation}\label{eq:cor3}
n_i(t,x) = \rho_i(t) \, \delta(x - \overline{x}) \; \mbox{ for a.e. } t \in [0,T], \; \mbox{ for } \; i=1,2,3.
\end{equation}
\end{cor}
\begin{proof}
For $M=3$, under the additional assumption~\eqref{ass.acor3}, the result given by~\eqref{eq:cor3} is a straightforward consequence of the general asymptotic result~\eqref{eq:suppn1_C3}.
\begin{flushright}
\hfill $\Box$
\end{flushright}
\end{proof}

\begin{cor}[Selection of multiple clones] 
\label{cor5}
Assume $M=3$ and let the assumptions of Theorem \ref{th:case3} along with the additional assumption~\eqref{ass.acor5} hold. Then, the measures $n_1$, $n_2$ and $n_3$ given by Theorem \ref{th:case3} are such that
\begin{equation}
\label{eq:cor5}
n_{i}(t,x) = \sum_{j=1}^N  \rho_{ij}(t) \delta(x - \overline{x}_j)  \quad \mbox{for a.e. } \; t \in [0,T], \; \mbox{ for } \; i=1,2,3.
\end{equation}
\end{cor}

\begin{proof}
For $M=3$, under the additional assumption~\eqref{ass.acor5}, the result given by~\eqref{eq:cor5} is a straightforward consequence of the general asymptotic result~\eqref{eq:suppn1_C3}.
\begin{flushright}
\hfill $\Box$
\end{flushright}
\end{proof}

\begin{cor}[Absence of clonal selection]
\label{cor4}
Assume $M=3$ and let the assumptions of Theorem \ref{th:case3} along with the additional assumption~\eqref{ass.acor4ap} hold. Moreover, let $M=3$. Then, the measures $n_1$, $n_2$ and $n_3$ given by Theorem \ref{th:case3} are such that
\begin{equation}\label{eq:cor4c}
\supp(n_i(t, \cdot)) = [0,1]  \; \mbox{ for a.e. }  t \in [0,T], \; \mbox{ for } \; i=1,2,3.
\end{equation}
\end{cor}

\begin{proof}
For $M=3$, under the additional assumption~\eqref{ass.acor4ap}, the result given by~\eqref{eq:cor4c} is a straightforward consequence of the general asymptotic result~\eqref{eq:suppn1_C3}.
\begin{flushright}
\hfill $\Box$
\end{flushright}

\end{proof}

\section{Numerical solutions}
We integrate the asymptotic results established by Corollaries~\ref{cor3}~-~\ref{cor4} of Theorem \ref{th:case3} with numerical solutions of the Cauchy problem defined by the system of IDEs~\eqref{e.mod3o} with $M=3$ complemented with biologically relevant initial conditions (\emph{vid.} Section \ref{subsec:resu}). We parameterise the model based on patient data from the existing literature (\emph{vid.} Section \ref{subsec:modparam}).

\subsection{Setup of numerical simulations and model calibration}\label{subsec:modparam}
To construct numerical solutions, we choose a uniform discretisation of the interval $[0,1]$ that consists of 1000 points. We assume $t \in [0,T]$ and we approximate the IDE system for the population density functions $n_1(t,x)$, $n_2(t,x)$ and $n_3(t,x)$ using the forward Euler method with step size $10^{-4}$. We choose $T=10^{4}$. Under the parameter settings considered here, such a value of $T$ corresponds to a time frame of approximatively 30 years after the appearance of the first leukemic clones. This is biologically reasonable since acute leukemia has a pronounced peak of incidence in late adulthood. 

Numerical computations are performed in {\sc Matlab} for two main parameter settings: one under which the total cell density functions $\rho_{1}(t)$, $\rho_{2}(t)$ and $\rho_{3}(t)$ converge to some stable values ({\it vid.} Section~\ref{subsec:modparamori}), and the other such that the total cell density functions undergo oscillations ({\it vid.} Section~\ref{subsec:modparamhopf}).

\subsubsection{Model calibration 1}\label{subsec:modparamori}
To model an initial scenario whereby, due to clonal heterogeneity, all possible leukemic clones are present in small numbers in every maturation compartment, and the total densities of healthy cells are close to the cell counts at physiological equilibrium, we use the following initial data 
\begin{equation}
\label{e.param8}
n^0_i(x) = N_i \exp\left( -\frac{x^2}{0.2}\right) \; \mbox{ for } \; i=1,2,3,
\end{equation}
with 
\begin{equation}
\label{e.param8b}
N_1 \approx 2.5 \times 10^7, \quad N_2 \approx 3.8 \times 10^9 \quad \text{and} \quad N_3 \approx 10^8.
\end{equation}
Such initial conditions satisfy assumptions~\eqref{e.mod3IC}.

Multi-compartment models of hematopoiesis after bone marrow transplantation allow to estimate the proliferation rates and the self-renewal fractions of healthy stem and progenitor cells, as well as the parameters of the feedback signal \citep{StiehlBaranHoMarciniakCR,StiehlBMT}. In particular, coherently with the estimations performed in~\citep{StiehlBMT}, we assume that
\begin{equation}\label{e.param1}
a_1(0)=0.85, \;\; a_2(0)=0.84, \;\; p_1(0)=0.1/day, \;\; p_2(0)=0.4/day 
\end{equation}
and
\begin{equation}\label{e.param1b}
K=1.75 \times 10^{-9}kg/cell.
\end{equation}
Clearance rates of mature cells and leukemic blasts can be estimated based on patient data and they appear to be between $0.1/day$ and $2.3/day$~\citep{Cartwright,Malinowska,Savitskiy}. For this reason, we choose 
\begin{equation}\label{e.param2}
d=2/day.
\end{equation}
In agreement with the estimations performed in~\citep{StiehlBaranHoMarciniakCR,StiehlBMT}, which are based on clinical data of blast dynamics in relapsing patients, we impose the following conditions 
\begin{equation}\label{e.param3}
0.85 \leq a_1(x) \leq 0.99 \; \mbox{ and } \; 0.1/day \, \leq p_1(x) < 1/day \; \mbox{ for all } \; x \in (0,1].
\end{equation}
The values of the function $a_2(x)$ are constrained by assumption \eqref{ass.a1i}. Moreover, since it is well accepted that stem cells divide at lower rates compared to progenitor cells~\citep{Adams,Cronkite,Shepherd}, we impose the following additional conditions
\begin{equation}\label{ass.p12}
0 < p_1(x) \leq p_2(x) < 1 \quad \mbox{for all } \; x \in [0,1].
\end{equation}

We let the continuous structuring variable $x$ parameterise the cell proliferation rate and we use of the following definitions 
\begin{equation}\label{e.param4}
p_1(x) = \alpha_1 + \beta_1 \, x \quad \mbox{and} \quad p_2(x) = \alpha_2 + \beta_2 \, x.
\end{equation}
In order to fulfil conditions \eqref{e.param1}, we choose 
\begin{equation}\label{e.param4alpha}
\alpha_1=0.1/day \quad \text{and} \quad \alpha_2=0.4/day.
\end{equation}
Moreover, we choose 
\begin{equation}\label{e.param4beta}
\beta_1=0.2/day \quad \text{and} \quad \beta_2=0.5/day
\end{equation}
so that conditions~\eqref{e.param3} and \eqref{ass.p12} are satisfied. To take into account the complex relationship between proliferation and self-renewal observed in leukemic cells~\citep{Doulatov,Ghosh,Kikushige,Lagunas-Rangel,Wang2,Yassin}, we assume the correspondence between the self-renewal fraction and the cell proliferation rate to be non-bijective and among all possible definitions which satisfy conditions~\eqref{ass.a1i}, \eqref{e.param1} and \eqref{e.param3} we choose 
\begin{equation}\label{e.param5b}
a_2(x) = 0.5 \; \exp\left[-\frac{(x-0.4)^2}{8.82}\right] + 0.349
\end{equation}
and
\begin{equation}\label{e.param5}
a_1(x) = \exp\left[ -\frac{(x-0.6)^2}{9.68}\right] - 0.1135
\end{equation}
or
\begin{equation}\label{e.param6}
a_1(x) = a^0_1 + 0.1 \, \sum_{j=1}^4 \, \exp\left[ -\frac{(x-\overline{x}_j)^2}{0.0025}\right]
\end{equation}
with
\begin{equation}\label{e.param6b}
\overline{x}_j \in \{0.35, 0.55, 0.7, 0.85\} \quad \mbox{and} \quad a^0_1 = 0.85 - 0.1 \, \sum_{j=1}^4 \, \exp\left[ -\frac{\overline{x}_j^2}{0.0025}\right],
\end{equation}  
or
\begin{equation}\label{e.param7}
a_1(x) \equiv 0.88.
\end{equation}
Definition~\eqref{e.param5b} models a biological scenario where the leukemic clone identified by $x=0.4$ has the highest self-renewal fraction among progenitor cells. Moreover, definitions~\eqref{e.param5}, \eqref{e.param6}-\eqref{e.param6b}, and \eqref{e.param7} model, respectively, three distinct situations whereby, due to differential gene expression at different maturation stages:
\begin{itemize}
\item[-] the leukemic clone identified by $x=0.6$ has the highest self-renewal fraction among stem cells;
\item[]
\item[-] the leukemic clones corresponding to $x=0.35$, $x=0.55$, $x=0.7$ and $x=0.85$ have the highest self-renewal fraction among stem cells;
\item[]
\item[-] all leukemic clones in the stem-cell compartment have the same self-renewal fraction.
\end{itemize}

\subsubsection{Model calibration 2}\label{subsec:modparamhopf}
On the basis of considerations analogous to those presented in Section~\ref{subsec:modparamori}, we use the initial data~\eqref{e.param8} with 
\begin{equation}
\label{e.param8hopf}
N_1 \approx 4.37 \times 10^6, \quad N_2 \approx 5 \times 10^8, \quad N_3 \approx 4.28 \times 10^8
\end{equation}
and we choose
\begin{equation}\label{e.param1bhopf}
K=1.75 \times 10^{-9}kg/cell, \quad d=0.15/day.
\end{equation}
Moreover, we define the functions $p_1(x)$ and $p_2(x)$ according to~\eqref{e.param4} with 
\begin{equation}\label{e.param4hopfalpha}
\alpha_1=0.975/day, \quad \alpha_2=0.04/day
\end{equation}
and
\begin{equation}\label{e.param4hopfbeta}
\beta_1=0.025/day, \quad \beta_2=0.0333/day.
\end{equation}
Finally, we assume
\begin{equation}\label{e.param5hopf}
a_1(x) = \frac{0.7}{0.8865} \, \left\{\exp\left[ -\frac{(x-0.6)^2}{9.68}\right]-0.1135\right\}
\end{equation}
and
\begin{equation}\label{e.param5bhopf}
a_2(x) = \frac{0.6}{0.8467} \, \left\{0.5 \; \exp\left[-\frac{(x-0.4)^2}{8.82}\right] + 0.349\right\},
\end{equation}
\emph{i.e.} we consider a biological scenario whereby the leukemic clones identified by $x=0.6$ and $x=0.4$ have the highest self-renewal fraction among stem cells and progenitor cells, respectively. The parameters and functions~\eqref{e.param4hopfalpha}-\eqref{e.param5bhopf} are such that the values of the functions $p_1(x)$, $p_2(x)$, $a_1(x)$ and $a_2(x)$ at the point $x=0.6$ [\emph{i.e.} the maximum point of the function $a_1(x)$] coincide with the parameter values for which the solutions of the ODE system \eqref{ODE} with $M=3$ are known to undergo periodic  oscillations \citep{Knauer2012,Knaueretal2019}.

\subsection{Main results} \label{subsec:resu}
In agreement with the asymptotic results established by Corollary \ref{cor3} and Corollary \ref{cor5} of Theorem \ref{th:case3}, the numerical solutions presented in Fig.~\ref{fig:f11} and Fig.~\ref{fig:f21} show that, under the parameter setting given in Section~\ref{subsec:modparamori}, the population density functions of stem cells $n_1(t,x)$, progenitor cells $n_2(t,x)$ and mature cells/leukemic blasts $n_3(t,x)$ become progressively concentrated at the maximum point(s) of the function $a_1(x)$. Moreover, the plots in the insets of Fig.~\ref{fig:f11} and Fig.~\ref{fig:f21} highlight the existence of leukemic clones which grow transiently before becoming ultimately extinct. Finally, the numerical solutions displayed in Fig.~\ref{fig:f31} illustrate how, in agreement with the asymptotic results of Corollary~\ref{cor4} of Theorem~\ref{th:case3}, if the function $a_1(x)$ is constant, the long-term limits of the population density functions $n_1(t,x)$, $n_2(t,x)$ and $n_3(t,x)$ are not concentrated at any particular point. 

As mentioned earlier in the paper, these results communicate the biological notion that the self-renewal fraction of leukemic stem cells determines the fate of clonal selection. In fact, the leukemic clones characterised by the highest stem cell self-renewal fraction are selected, regardless of their properties in terms of stem cell proliferation rate, progenitor cell proliferation rate and progenitor cell self-renewal fraction. This supports the idea that interclonal variability in the self-renewal fraction of leukemic stem cells provides the necessary substrate for clonal selection to act upon. 
\begin{figure}
\centering
\includegraphics[width=1\linewidth]{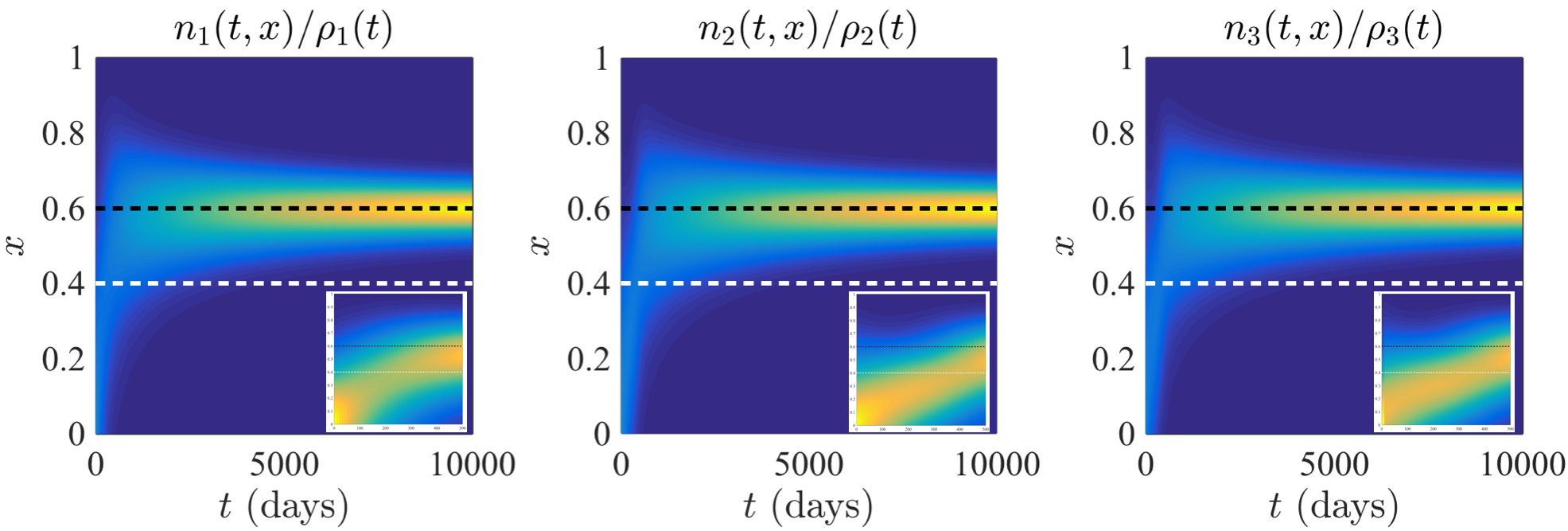}
\caption{{\bf Selection of one single clone.} Dynamics of the normalised cell population density functions $n_{1}(t,x)/\rho_1(t)$ (left panel), $n_{2}(t,x)/\rho_2(t)$ (central panel) and $n_{3}(t,x)/\rho_3(t)$ (right panel) under the parameter setting given in Section~\ref{subsec:modparamori} with $a_1(x)$ defined according to \eqref{e.param5}. The black lines highlight the clone with the maximum self-renewal fraction in the stem-cell compartment [\emph{i.e.} the maximum point of the function $a_1(x)$], while the white lines highlight the clone with the maximum self-renewal fraction in the progenitor-cell compartment [\emph{i.e.} the maximum point of the function $a_2(x)$]. The transient behaviour of the solutions from the initial conditions is displayed by the plots in the insets, which show the dynamics of the normalised population density functions for $t \in [0,500]$. The colour scale ranges from blue (low density) to yellow (high density).  
}
\label{fig:f11}
\end{figure}
\begin{figure}
\centering
\includegraphics[width=1\linewidth]{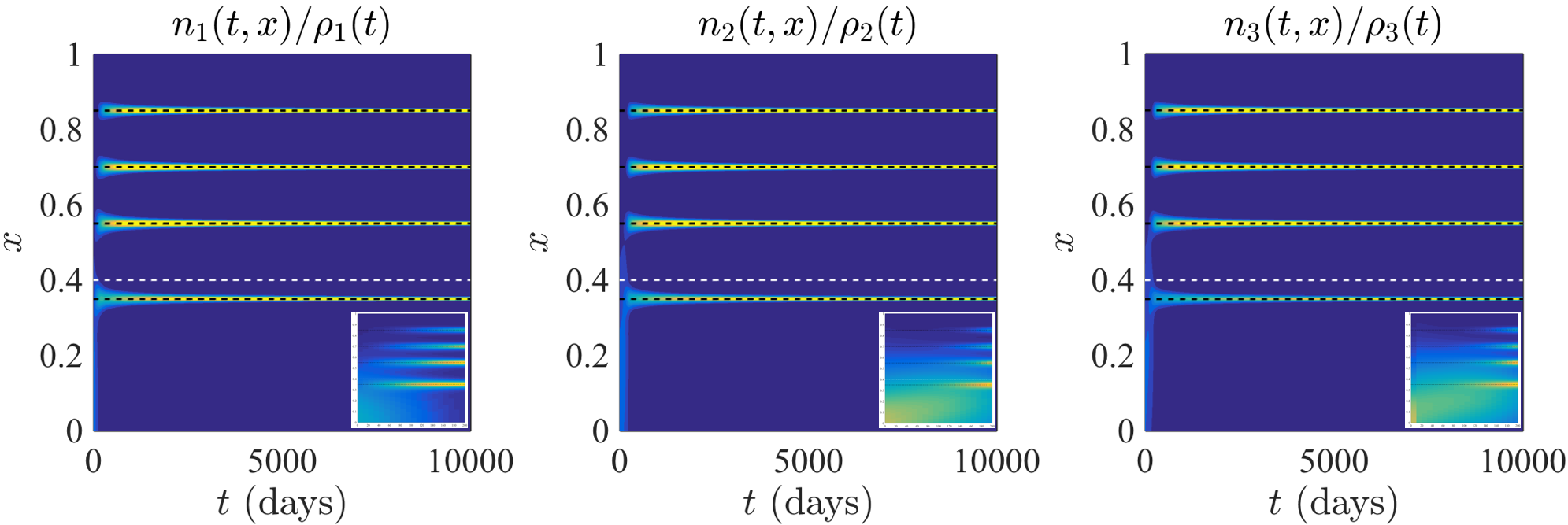}
\caption{{\bf Selection of multiple clones.} Dynamics of the normalised cell population density functions $n_{1}(t,x)/\rho_1(t)$ (left panel), $n_{2}(t,x)/\rho_2(t)$ (central panel) and $n_{3}(t,x)/\rho_3(t)$ (right panel) under the parameter setting given in Section~\ref{subsec:modparamori} with $a_1(x)$ defined according to~\eqref{e.param6} and \eqref{e.param6b}. The black lines highlight the clones with the maximum self-renewal fraction in the stem-cell compartment [\emph{i.e.} the maximum points of the function $a_1(x)$], while the white lines highlight the clone with the maximum self-renewal fraction in the progenitor-cell compartment [\emph{i.e.} the maximum point of the function $a_2(x)$]. The transient behaviour of the solutions from the initial conditions is displayed by the plots in the insets, which show the dynamics of the normalised population density functions for $t \in [0,200]$. The colour scale ranges from blue (low density) to yellow (high density).  
}
\label{fig:f21}
\end{figure}
\begin{figure}
\centering
\includegraphics[width=1\linewidth]{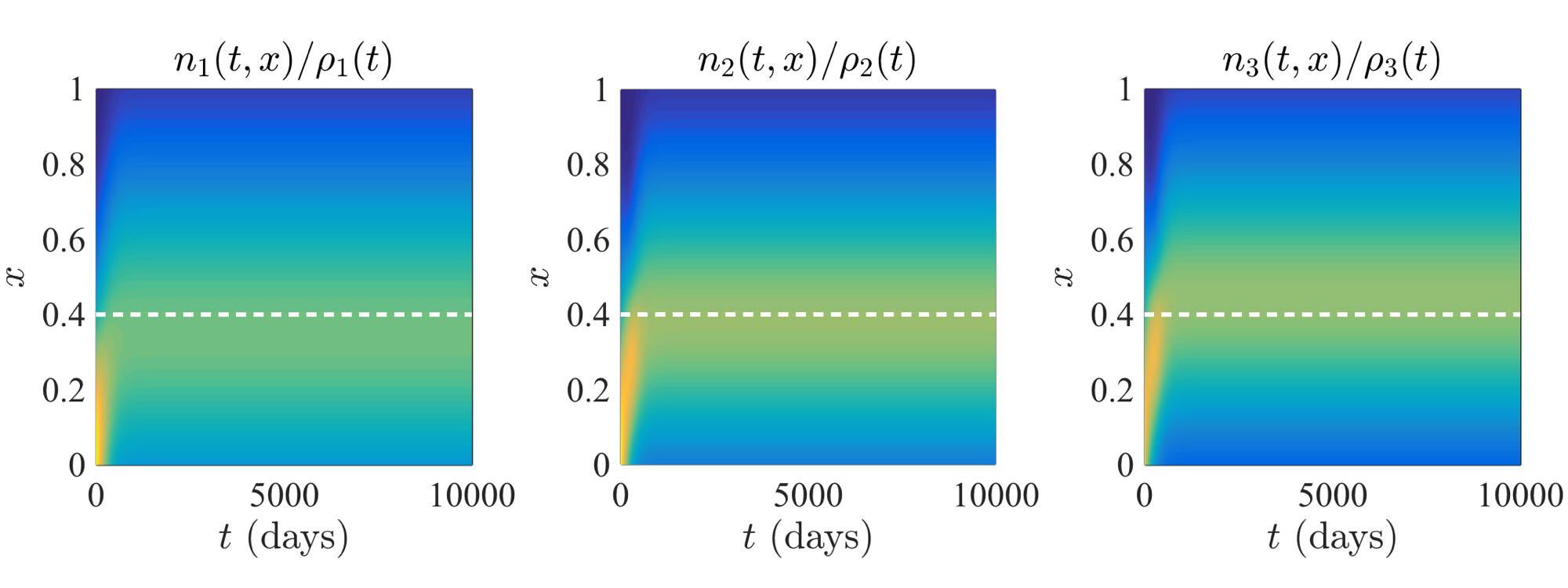}
\caption{{\bf Absence of clonal selection.} Dynamics of the normalised cell population density functions $n_{1}(t,x)/\rho_1(t)$ (left panel), $n_{2}(t,x)/\rho_2(t)$ (central panel) and $n_{3}(t,x)/\rho_3(t)$ (right panel) under the parameter setting given in Section~\ref{subsec:modparamori} with $a_1(x)$ defined according to~\eqref{e.param7}. The white lines highlight the clone with the maximum self-renewal fraction in the progenitor-cell compartment [\emph{i.e.} the maximum point of the function $a_2(x)$]. The colour scale ranges from blue (low density) to yellow (high density).  
}
\label{fig:f31}
\end{figure}
\begin{figure}
\centering
\includegraphics[width=1\linewidth]{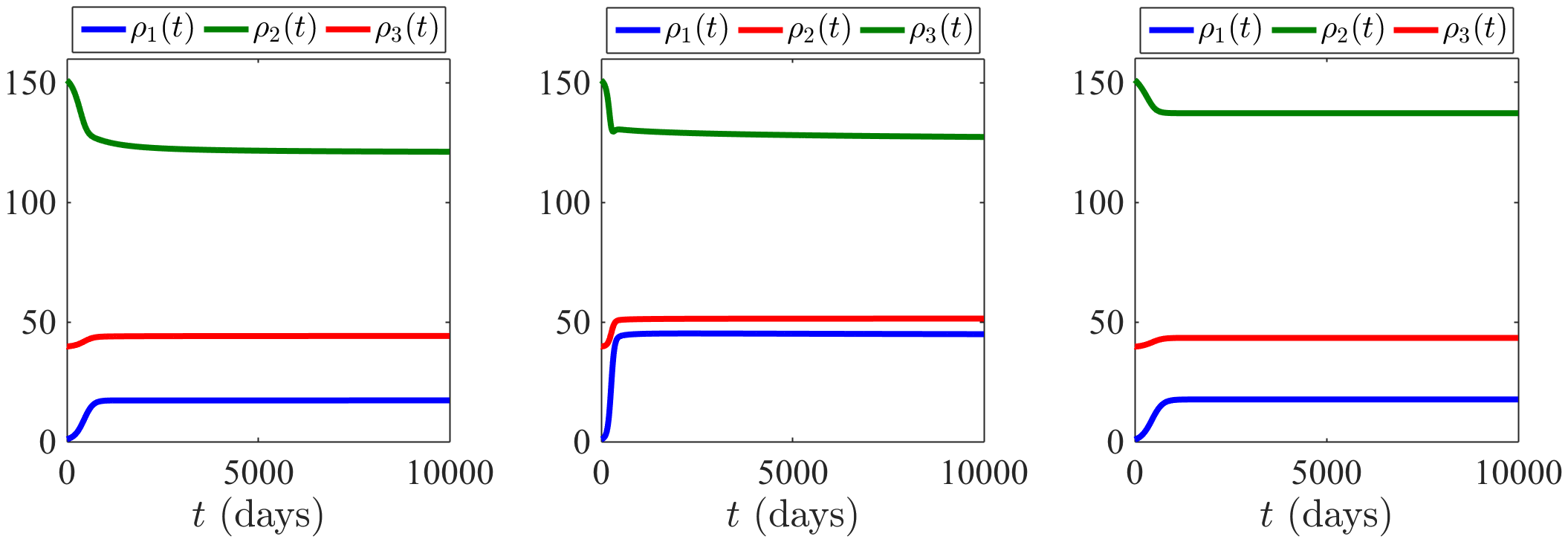}
\caption{{\bf Dynamics of the total cell densities.} {\it Left panel.} Dynamics of the total density of stem cells $\rho_{1}(t)$ (blue line), progenitor cells $\rho_{2}(t)$ (red line) and mature cells/leukemic blasts $\rho_{3}(t)$ (red line) under the parameter setting given in Section~\ref{subsec:modparamori} with $a_1(x)$ defined according to~\eqref{e.param5} (left panel), \eqref{e.param6}-\eqref{e.param6b} (central panel) and~\eqref{e.param7} (right panel). Values are in units of $10^7$.}
\label{fig:f1232}
\end{figure}

Under the parameter setting given in Section~\ref{subsec:modparamori}, the total density of stem cells $\rho_{1}(t)$, progenitor cells $\rho_{2}(t)$ and mature cells/leukemic blasts $\rho_{3}(t)$ converge to some stable values (\emph{vid.} Fig.~\ref{fig:f1232}). However, it is known that for given parameter choices the solutions of the ODE system \eqref{ODE} with $M=3$ undergo periodic oscillations, which result from the occurrence of Hopf bifurcation \citep{Knauer2012,Knaueretal2019}.  The numerical results presented in Fig.~\ref{fig:Hopf1rho} show that, under the parameter setting given in Section~\ref{subsec:modparamhopf}, oscillations emerge in the integrals of the solutions to the IDE system \eqref{e.mod3o} with $M=3$ [\emph{i.e.} in the total cell densities $\rho_1(t)$, $\rho_2(t)$ and $\rho_3(t)$]. In analogy with the previous cases, the numerical solutions presented in Fig.~\ref{fig:Hopf1n} indicate that the cell population density functions $n_1(t,x)$, $n_2(t,x)$ and $n_3(t,x)$ become progressively concentrated at the maximum point of the function $a_1(x)$ -- \emph{i.e.} the leukemic clone characterised by the highest stem cell self-renewal fraction is ultimately selected.
\begin{figure}
\centering
\includegraphics[width=1\linewidth]{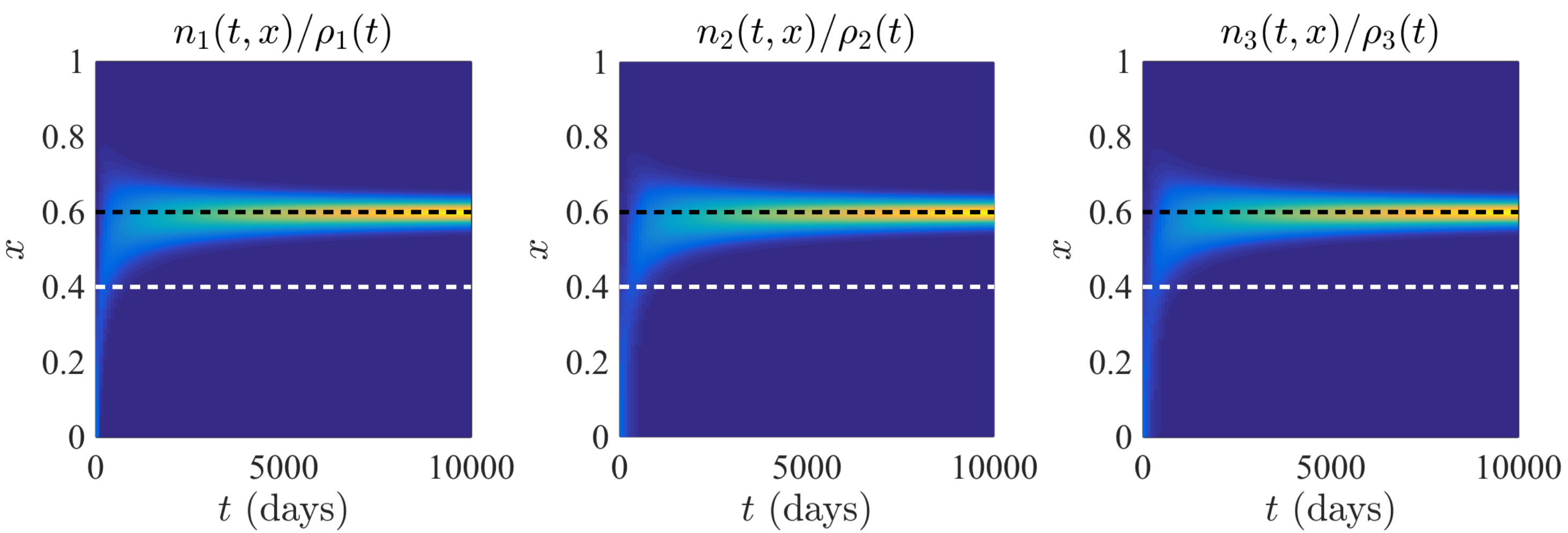}
\caption{{\bf Clonal selection when the total cell densities undergo oscillations.} Dynamics of the normalised cell population density functions $n_{1}(t,x)/\rho_1(t)$ (left panel), $n_{2}(t,x)/\rho_2(t)$ (central panel) and $n_{3}(t,x)/\rho_3(t)$ (right panel) under the parameter setting given in Section~\ref{subsec:modparamhopf}. The black lines highlight the clone with the maximum self-renewal fraction in the stem-cell compartment [\emph{i.e.} the maximum point of the function $a_1(x)$], while the white lines highlight the clone with the maximum self-renewal fraction in the progenitor-cell compartment [\emph{i.e.} the maximum point of the function $a_2(x)$]. The colour scale ranges from blue (low density) to yellow (high density).} 
\label{fig:Hopf1n}
\end{figure}
\begin{figure}
\centering
\includegraphics[width=1\linewidth]{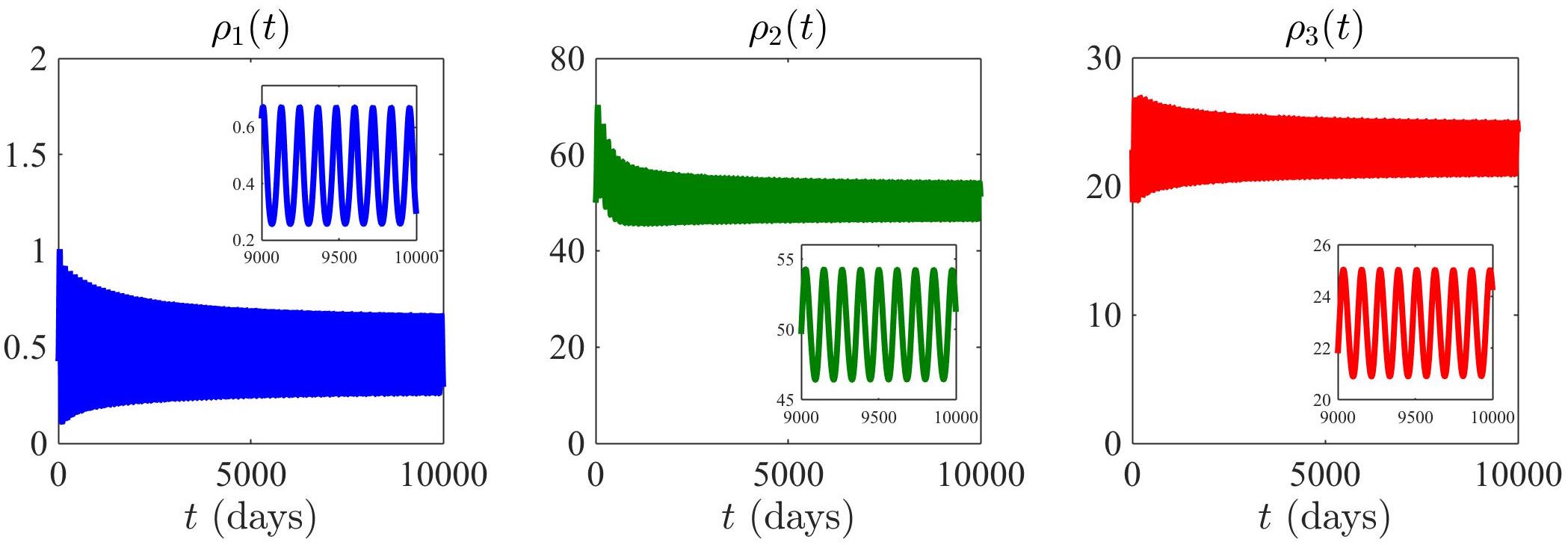}
\caption{{\bf Emergence of oscillations in the total cell densities.} Dynamics of the total density of stem cells $\rho_{1}(t)$ (blue line), progenitor cells $\rho_{2}(t)$ (red line) and mature cells/leukemic blasts $\rho_{3}(t)$ (red line) under the parameter setting given in Section~\ref{subsec:modparamhopf}. Values are in units of $10^7$.
}
\label{fig:Hopf1rho}
\end{figure}

\newpage
\section{Discussion and conclusions}
Recent progress in genetic techniques has shed light on the complex co-evolution of malignant cell clones in leukemias \citep{Anderson,Belderbos,Ding,Ley}. However, several aspects of clonal selection still remain unclear. In this work, we have adopted a mathematical modelling approach to study clonal selection in acute leukemias, with the aim of supporting a better understanding of the biological mechanisms which underpin observable clonal dynamics.

Our model consists of a system of coupled IDEs that describe the dynamics of cells at different maturation stages -- from stem cells through increasingly mature progenitor cells to mature-cells/blasts -- seen as distinct compartments. Each maturation compartment is structured by a continuous variable that identifies the clone of the cells. In order to incorporate into our model the high degree of interclonal heterogeneity which is observed in leukemia patients, we let the cellular proliferation rate and self-renewal fraction (\emph{i.e.} the fraction of progeny cells adopting the same fate as their parent cell) in each maturation compartment be functions of the structuring variable.

In the framework of this model, we have established a number of analytical results which give answers to the open questions {\bf Q1}-{\bf Q3} posed in the introduction of this paper. In summary:
\begin{itemize}
\item [{\bf A1}] {\it Clonal selection is driven by the self-renewal fraction of leukemic stem cells.} 
Theorem \ref{th:case3} rigorously shows that all leukemic clones with a non-maximal stem cell self-renewal fraction ultimately become extinct, independently of their proliferation rate.
\item [{\bf A2}] {\it Non-stem cell properties do not have a substantial impact on clonal selection.} The result established by Corollary \ref{cor4} formalises the idea that, in a scenario where the stem cells of all clones have the same self-renewal fraction, one should expect the stable coexistence of all clones to occur. 
\item [{\bf A3}] {\it Only the clones whose stem cells are endowed with the highest self-renewal fraction can stably coexist in the presence of interclonal heterogeneity.} Corollaries \ref{cor3} and \ref{cor5} put on a rigorous basis the notion that the leukemic clones whose stem cells have the highest self-renewal fraction are ultimately selected.
\end{itemize}

We have integrated our asymptotic results with numerical solutions of a calibrated version of the model based on real patient data. In agreement with the theoretical results by \citet{StiehlBaranHoMarciniak}, our numerical results reveal the existence of leukemic clones which display transient growth before becoming extinct. Such an emergent behaviour was not captured by the IDE model considered by \citet{Busse}, who studied the case where cells of different clones at the same maturation stage have the same proliferation rate. This suggests that interclonal heterogeneity in the cell proliferation rate may have an impact on transient clonal dynamics.   

It has been shown using genetic techniques that, in most leukemia patients, the majority of leukemic cells is derived from a relatively small number of clones \citep{Anderson,Ding}. Our results indicate that this may be due to the fact that only a few leukemic clones are characterised by a high stem cell self-renewal fraction. 

There is accumulating experimental evidence that many of the different genetic mutations involved in the development of leukemia increase the cell self-renewal. Possible examples include the TIM-3/Gal-9 autocrine stimulation \citep{Kikushige} or alterations of Wnt/$\beta$-Catenin signalling \citep{Wang2}. Moreover, it has been shown that genetic mutations can simultaneously affect the self-renewal and proliferation of cells due to crosstalk between different signalling pathways. For instance, the NUP98-D{\rm d}x10 oncogene increases both cell proliferation and self-renewal \citep{Yassin}; hyperactivation of the mTOR pathway leads to S6K1-mediated increase in self-renewal and reduction in proliferation \citep{Ghosh}; up-regulation of PLZF brings about increased self-renewal and reduced proliferation \citep{Doulatov}. The outcome of our analysis, which ascribes a pivotal role to increased self-renewal in orchestrating clonal selection, is in line with the observation that all the aforementioned genetic alterations lead to increased self-renewal, whereas they have divergent effects on cell proliferation.

The result that clonal selection is not influenced by progenitor cell properties is biologically meaningful, since it implies that mutations which affect only the properties of progenitor cells without altering the properties of stem cells cannot lead to the selection of leukemic clones. This result is new and not self-evident, as progenitor cells can expand independently of the influx from the stem cell compartment -- although they possess smaller self-renewal fractions than stem cells \citep{Roelofs,StiehlBMT,Yamamoto}. This insight is also clinically relevant. In fact, although it is known that progenitor cells are more sensitive to treatment interventions than stem cells, which are located in protective niches, our results suggest that manipulations of progenitor cells have no impact on clonal selection phenomena.

A rigorous mathematical understanding of clonal selection is needed since the potentially nonlinear interplay between different genetic and epigenetic hits \citep{Bacher,Gale,Heuser} leads to complex fitness landscapes and non-trivial interdependencies between self-renewal fraction and proliferation rate of leukemic cells. In particular, there is evidence that combinations of leukemic mutations occur frequently in patients \citep{Naoe,Pui}. An \emph{in silico} approach can help to disentangle the impact of different mutations on the properties of leukemic stem cells. In this regard, our modelling approach can be further developed in several directions. For instance, in line with what was done for the ODE counterpart of the model presented here \citep{StiehlBaranHoMarciniak,StiehlBD,StiehlSciRep}, we plan to extend our model to take into account the effect of multiple feedback mechanisms and incorporate the occurrence of \emph{de novo} mutations. Moreover, along the lines of the modelling method proposed by Doumic \emph{et al.} \citep{Doumic,GJMC2012}, it may be interesting to replace the discrete maturation structure considered in this work by a continuous age structure, which would lead to the definition of a fully-continuously structured population model of clonal selection in acute leukemias. From a mathematical point of view, we also plan to carry out a systematic investigation of the conditions for the emergence of oscillations in the total cell densities, as shown by our numerical solutions.  

\newpage
\appendix
\section{Multi-compartmental ODE model}
\label{sec:ode}
In a number of previous papers \citep{StiehlBaranHoMarciniak,StiehlBaranHoMarciniakCR,StiehlSciRep,StiehlBD,StiehlMMNP}, it was shown that mathematical models defined in the framework of the following ODE system can effectively recapitulate clinical data from leukemia patients
\begin{equation}
\label{ODE}
\left\{
\begin{array}{ll}
\displaystyle{\frac{{\rm d}}{{\rm d} t} N^j_1(t) =  \left(\frac{2 \, a^j_{1}}{1 + K \, Z_M(t)} - 1 \right) \, p^j_1 \, N^j_1(t),}
\\\\
\displaystyle{\frac{{\rm d}}{{\rm d} t} N^j_i(t) = 2 \, \left(1 - \frac{a^j_{i-1}}{1 + K \, Z_M(t)} \right) \, p^j_{i-1} \, N^j_{i-1}(t)} 
\\
\displaystyle{\phantom{\frac{{\rm d}}{{\rm d} t} N^j_i(t) =} + \left(\frac{2 \, a^j_{i}}{1 + K \, Z_M(t)}  - 1 \right) \, p^j_i \, N^j_i(t),}
\\\\
\displaystyle{\frac{{\rm d}}{{\rm d} t} N^j_M(t) = 2 \, \left(1 -  \frac{a^j_{M-1}}{1 + K \, Z_M(t)} \right) \, p^j_{M-1} \, N^j_{M-1}(t) - d \, N^j_M(t),}
\end{array}
\right.
\end{equation}
with $i=2, \ldots, M-1$, $j=0, \ldots, J$ and $\displaystyle{Z_M(t) = \displaystyle{\sum_{j=0}^J} N^j_M(t)}$.

The index $i$ denotes the cell maturation stage while the index $j$ indicates to which leukemic clone the cells belong. In particular, the stem-cell compartment is labelled by the index $i=1$, the indices $i=2, \ldots, M-1$ correspond to increasingly mature progenitor-cell compartments and the mature-cell/blast compartment is labelled by the index $i=M$. Moreover, the index $j=0$ refers to healthy cells, whereas the different leukemic clones are labelled by the indices $j=1, \ldots, J$. 

In the system of ODEs~\eqref{ODE}, the function $N^j_i(t)$ models the density of cells of clone $j$ at the maturation stage $i$ at the time instant $t \geq 0$. Cells in the compartment $i=M$ do not divide and are cleared from the system at rate $d>0$, which is assumed to be the same for healthy and leukemic cells~\citep{Busse,Malinowska,Savitskiy,StiehlBaranHoMarciniak, StiehlSciRep}. The parameters $p^j_i > 0$ and $a^j_i > 0$ model, respectively, the proliferation rate and the self-renewal fraction of cells of clone $j$ at the maturation stage $i$. Coherently with biological findings \citep{Kondo,Layton,Shinjo}, the terms $a^j_i$ are multiplied by the factor $\frac{1}{1 + K Z_M(t)}$ to model the fact that the signal which promotes the self-renewal of dividing cells is absorbed by mature cells and leukemic blasts at a rate that depends on the total density of these cells $Z_M(t)$. The parameter $K>0$ is related to the degradation rate of the feedback signal by mature cells and leukemic blasts. This has proved to be a biologically consistent way of modelling the effects of feedback signals that control cell self-renewal~\citep{MarciniakStiehl,StiehlBMT,StiehlSciRep,StiehlMCM,StiehlMMNP}. In principle, the effects of feedback signals that control cell proliferation could also be included. However, it has been demonstrated that such signals have little impact on the dynamics of the blood system \citep{MarciniakStiehl,StiehlAdvExp,StiehlBMT}. 

A version of the ODE model \eqref{ODE} with only one leukemic clone and three maturation stages (\emph{i.e.} $M=3$ and $J=1$) has been fully analysed by \citet{StiehlMMNP}, while a two compartmental version of the model for healthy hematopoiesis (\emph{i.e.} $M=2$ and $J=0$) has been analysed by \citet{Getto,Nakata,StiehlMCM}. Possible applications of this model to clinical data can be found in the works by \citet{StiehlBaranHoMarciniak,StiehlBaranHoMarciniakCR}, whereas applications to healthy hematopoiesis are provided in the publications by \citet{StiehlAdvExp,StiehlBMT}. Finally, a version of this model with a continuous differentiation structure has been proposed and studied by~\citet{Doumic}. 

\section{Proof of the uniform upper bounds~\eqref{eq:UBfinal}}
In this appendix we prove the upper bounds~\eqref{eq:UBfinal} through a suitable development of the method of proof presented in~\citep{Busse}. Using the system of IDEs~\eqref{e.mod3}, straightforward calculations give the following ODEs
\begin{eqnarray}
\label{ODEF1}
\e \, \frac{{\rm d}}{{\rm d}t}\frac{\rho_{1 \, \e}}{\rho_{2 \, \e}} &=& \frac{1}{\rho_{2 \, \e}} \int_0^1 P_{1}(\rho_{M \e}(t),x)\, n_{1 \e}  \, {\rm d}x \nonumber
\\ 
&& - \frac{\rho_{1 \, \e}}{\rho^2_{2 \, \e}} \int_0^1 Q_{1}(\rho_{M \e}(t),x)\, n_{1 \e}  \, {\rm d}x - \frac{\rho_{1 \, \e}}{\rho^2_{2 \, \e}} \int_0^1 P_{2}(\rho_{M \e}(t),x)\, n_{2 \e} \, {\rm d}x,
\end{eqnarray}

\begin{eqnarray}
\label{ODEF2}
\e \, \frac{{\rm d}}{{\rm d}t}\frac{\rho_{i \, \e}}{\rho_{i+1 \, \e}} &=& \frac{1}{\rho_{i+1 \, \e}} \int_0^1 Q_{i-1}(\rho_{M \e}(t),x)\, n_{i-1 \e}  \, {\rm d}x + \frac{1}{\rho_{i+1 \, \e}} \int_0^1 P_{i}(\rho_{M \e}(t),x)\, n_{i \e} \, {\rm d}x \nonumber
\\ 
&& - \frac{\rho_{i \, \e}}{\rho^2_{i+1 \, \e}} \int_0^1 Q_{i}(\rho_{M \e}(t),x)\, n_{i \e}  \, {\rm d}x 
\nonumber
\\ 
&&
- \frac{\rho_{i \, \e}}{\rho^2_{i+1 \, \e}} \int_0^1 P_{i+1}(\rho_{M \e}(t),x)\, n_{i+1 \e} \, {\rm d}x, \quad \text{for } \; i = 2, \ldots, M-2
\end{eqnarray}
and
\begin{eqnarray}
\label{ODEF3}
\e \, \frac{{\rm d}}{{\rm d}t}\frac{\rho_{M-1 \, \e}}{\rho_{M \, \e}} &=& \frac{1}{\rho_{M \, \e}} \int_0^1 Q_{M-2}(\rho_{M \e}(t),x)\, n_{M-2 \e}  \, {\rm d}x + \frac{1}{\rho_{M \, \e}} \int_0^1 P_{M-1}(\rho_{M \e}(t),x)\, n_{M-1 \e} \, {\rm d}x \nonumber
\\ 
&& - \frac{\rho_{M-1 \, \e}}{\rho^2_{M \, \e}} \int_0^1 Q_{M-1}(\rho_{M \e}(t),x)\, n_{M-1 \e}  \, {\rm d}x + d \, \frac{\rho_{M-1 \, \e}}{\rho_{M \, \e}}.
\end{eqnarray}
Under assumptions~\eqref{ass.a}-\eqref{ass.a1i}, for all $i = 1, \ldots, M-1$ we have
$$
P_{i}(\rho_{M \, \e}(t),x) >  - \|p_{i}\|_{{\rm L}^\infty([0,1])}, \quad P_{i}(\rho_{M \, \e}(t),x) \leq 2 \, \|a_{i}\|_{{\rm L}^\infty([0,1])} \|p_{i}\|_{{\rm L}^\infty([0,1])} =: \overline{P}_i, 
$$
and
$$
Q_{i}(\rho_{M \, \e}(t),x) \geq 2 \, \left(1-\|a_{i}\|_{{\rm L}^\infty([0,1])}\right) \inf_{x \in [0,1]} p_{i} =: \underline{Q}_i  > 0, \quad Q_{i}(\rho_{M \, \e}(t),x) \leq 2 \, \|p_{i}\|_{{\rm L}^\infty([0,1])},
$$
for all $t \geq 0$ and all $x \in [0,1]$. Hence, estimating from above the right-hand sides of the ODEs~\eqref{ODEF1}-\eqref{ODEF3} using the above estimates on $P_i$ and $Q_i$ along with the non-negativity of $n_{i \, \e}$ and $\rho_{i \, \e}$ for all $i = 1, \ldots, M$ gives the following differential inequalities
\begin{eqnarray}
\label{DIEF1}
\e \, \frac{{\rm d}}{{\rm d}t}\frac{\rho_{1 \, \e}}{\rho_{2 \, \e}} &\leq& \overline{P}_1 \, \frac{\rho_{1 \, \e}}{\rho_{2 \, \e}} - \underline{Q}_1 \, \frac{\rho^2_{1 \, \e}}{\rho^2_{2 \, \e}} + \|p_{2}\|_{{\rm L}^\infty([0,1])} \, \frac{\rho_{1 \, \e}}{\rho_{2 \, \e}}\nonumber
\\
&\leq& \left[\left(\overline{P}_1 + \|p_{2}\|_{{\rm L}^\infty([0,1])}\right) - \underline{Q}_1 \frac{\rho_{1 \, \e}}{\rho_{2 \, \e}}\right] \frac{\rho_{1 \, \e}}{\rho_{2 \, \e}},
\end{eqnarray}

\begin{eqnarray}
\label{DIEF2}
\e \, \frac{{\rm d}}{{\rm d}t}\frac{\rho_{i \, \e}}{\rho_{i+1 \, \e}} &\leq& 2 \, \|p_{i-1}\|_{{\rm L}^\infty([0,1])} \, \frac{\rho_{i-1 \, \e}}{\rho_{i+1 \, \e}} + \overline{P}_i \, \frac{\rho_{i \, \e}}{\rho_{i+1 \, \e}}  - \underline{Q}_i \, \frac{\rho^2_{i \, \e}}{\rho^2_{i+1 \, \e}} + \|p_{i+1}\|_{{\rm L}^\infty([0,1])} \frac{\rho_{i \, \e}}{\rho_{i+1 \, \e}} \nonumber
\\
&\leq& \left[\left(2 \, \|p_{i-1}\|_{{\rm L}^\infty([0,1])} \frac{\rho_{i-1 \, \e}}{\rho_{i \, \e}} + \overline{P}_i + \|p_{i+1}\|_{{\rm L}^\infty([0,1])}\right) - \underline{Q}_i \frac{\rho_{i \, \e}}{\rho_{i+1 \, \e}}\right] \frac{\rho_{i \, \e}}{\rho_{i+1 \, \e}} \qquad
\end{eqnarray}
for $i = 2, \ldots, M-2$,

\begin{eqnarray}
\label{DIEF3}
\e \, \frac{{\rm d}}{{\rm d}t}\frac{\rho_{M-1 \, \e}}{\rho_{M \, \e}} &\leq& 2 \, \|p_{M-2}\|_{{\rm L}^\infty([0,1])} \, \frac{\rho_{M-2 \, \e}}{\rho_{M \, \e}} + \overline{P}_{M-1} \, \frac{\rho_{M-1 \, \e}}{\rho_{M \, \e}} - \underline{Q}_{M-1} \, \frac{\rho^2_{M-1 \, \e}}{\rho^2_{M \, \e}} + d \, \frac{\rho_{M-1 \, \e}}{\rho_{M \, \e}} \nonumber
\\
&\leq& \left[\left(2 \, \|p_{M-2}\|_{{\rm L}^\infty([0,1])} \frac{\rho_{M-2 \, \e}}{\rho_{M-1 \, \e}} + \overline{P}_{M-1} + d\right) - \underline{Q}_{M-1} \frac{\rho_{M-1 \, \e}}{\rho_{M \, \e}}\right] \frac{\rho_{M-1 \, \e}}{\rho_{M \, \e}}. \qquad
\end{eqnarray}
The differential inequality~\eqref{DIEF1} yields
\begin{equation}
\label{DIEF1res}
\frac{\rho_{1 \, \e}(t)}{\rho_{2 \, \e}(t)} \leq \max \left(\frac{\|n^0_{1}\|_{{\rm L}^\infty([0,1])}}{\|n^0_{2}\|_{{\rm L}^\infty([0,1])}}, \frac{\overline{P}_1 + \|p_{2}\|_{{\rm L}^\infty([0,1])}}{\underline{Q}_1} \right) =: B_1,
\end{equation}
for all $t \geq 0$. Substituting the estimate~\eqref{DIEF1res} into the differential inequality~\eqref{DIEF2} for $i=2$ we find
\begin{equation}
\label{DIEF2res}
\frac{\rho_{2 \, \e}(t)}{\rho_{3 \, \e}(t)} \leq \max \left(\frac{\|n^0_{2}\|_{{\rm L}^\infty([0,1])}}{\|n^0_{3}\|_{{\rm L}^\infty([0,1])}}, \frac{2 \, \|p_{1}\|_{{\rm L}^\infty([0,1])} B_1 + \overline{P}_2 + \|p_{3}\|_{{\rm L}^\infty([0,1])}}{\underline{Q}_2} \right) =: B_2,
\end{equation}
for all $t \geq 0$. In a similar way, substituting the estimate~\eqref{DIEF2res} into the differential inequality~\eqref{DIEF2} for $i=3$ gives
\begin{equation}
\label{DIEF3res}
\frac{\rho_{3 \, \e}(t)}{\rho_{4 \, \e}(t)} \leq \max \left(\frac{\|n^0_{3}\|_{{\rm L}^\infty([0,1])}}{\|n^0_{4}\|_{{\rm L}^\infty([0,1])}}, \frac{2 \, \|p_{2}\|_{{\rm L}^\infty([0,1])} B_2 + \overline{P}_3 + \|p_{4}\|_{{\rm L}^\infty([0,1])}}{\underline{Q}_3} \right) =: B_3,
\end{equation}
for all $t \geq 0$. Using a bootstrap argument based on the method of proof that we have used for the case $i=3$, one can prove that
\begin{equation}
\label{DIEFires}
\frac{\rho_{i \, \e}(t)}{\rho_{i+1 \, \e}(t)} \leq \max \left(\frac{\|n^0_{i}\|_{{\rm L}^\infty([0,1])}}{\|n^0_{i+1}\|_{{\rm L}^\infty([0,1])}}, \frac{2 \, \|p_{i-1}\|_{{\rm L}^\infty([0,1])} B_{i-1} + \overline{P}_i + \|p_{i+1}\|_{{\rm L}^\infty([0,1])}}{\underline{Q}_i} \right) =: B_i,
\end{equation}
for all $t \geq 0$ and for all $i=4, \ldots, M-2$. Finally, substituting the estimate~\eqref{DIEFires} with $i=M-2$ into the differential inequality~\eqref{DIEF3} we obtain 
\begin{equation}
\label{DIEFMm1res}
\frac{\rho_{M-1 \, \e}(t)}{\rho_{M \, \e}(t)} \leq \max \left(\frac{\|n^0_{M-1}\|_{{\rm L}^\infty([0,1])}}{\|n^0_{M}\|_{{\rm L}^\infty([0,1])}}, \frac{2 \, \|p_{M-2}\|_{{\rm L}^\infty([0,1])} B_{M-2} + \overline{P}_{M-1} + d}{\underline{Q}_{M-1}} \right) =: B_{M-1}
\end{equation}
for all $t \geq 0$. Combining the estimates~\eqref{DIEF1res}-\eqref{DIEFMm1res} yields
$$
\rho_{i \, \e}(t) \leq \rho_{M \, \e}(t) A_i \quad \text{with} \quad A_i := \prod_{k=i}^{M-1} B_k > 0 \quad \text{for } \; i=1,\ldots,M-1
$$
for all $t \geq 0$, which allows us to conclude that
$$
P_{i}(\rho_{M \e}(t),x) \leq \left(\frac{2 \, a_i(x)}{1 + \frac{K}{A_i} \rho_{i \e}(t)} - 1 \right) p_i(x) \; \mbox{ for } \; i=1, \ldots, M-1 
$$
for all $t \geq 0$ and all $x \in [0,1]$. The latter inequality ensures that $P_{i}(\rho_{M \e}(t),x)$ satisfies the following relations for all $t \geq 0$ and each $i=1, \ldots, M-1$ 
%\begin{equation}
%\label{rhoiubtempnew}
%\rho_{i \e}(t) \leq \frac{A_i}{K} \, \left(2 \, \|a_{i}\|_{{\rm L}^\infty([0,1])}  -1 \right)
%\end{equation}
%or 
\begin{eqnarray}
\label{Piub1}
&&\text{if } \; \displaystyle{\rho_{i \e}(t) \leq \frac{A_i}{K} \, \left(2 \, \|a_{i}\|_{{\rm L}^\infty([0,1])}  -1 \right)} \; \text{ then } \nonumber
\\
&& \phantom{\text{if } \; \displaystyle{\left(2 \, \|a_{i}\|_{{\rm L}^\infty([0,1])} \right)}} 0 \leq \|P_{i}(\rho_{M \e}(t),\cdot)\|_{{\rm L}^\infty([0,1])} \leq \left(\frac{2 \, \|a_{i}\|_{{\rm L}^\infty([0,1])}}{1 + \frac{K}{A_i} \rho_{i \e}(t)} - 1 \right) \, \|p_{i}\|_{{\rm L}^\infty([0,1])} \quad \quad
\end{eqnarray}
while
\begin{eqnarray}
\label{Piub2}
&&\text{if } \; \displaystyle{\rho_{i \e}(t) > \frac{A_i}{K} \, \left(2 \, \|a_{i}\|_{{\rm L}^\infty([0,1])}  -1 \right)} \; \text{ then } \nonumber
\\
&& \phantom{\text{if } \; \displaystyle{\left(2 \, \|a_{i}\|_{{\rm L}^\infty([0,1])} \right)}} \|P_{i}(\rho_{M \e}(t),\cdot)\|_{{\rm L}^\infty([0,1])} \leq \left(\frac{2 \, \|a_{i}\|_{{\rm L}^\infty([0,1])}}{1 + \frac{K}{A_i} \rho_{i \e}(t)} - 1 \right) \, \inf_{x \in [0,1]} p_{i}  < 0. \quad \quad \quad
\end{eqnarray}

Integrating over $[0,1]$ both sides of the IDEs~\eqref{e.mod3} for $n_{1 \, \e}$ and estimating from above gives the following differential inequality
\begin{equation}
\label{finub1}
\e \, \frac{{\rm d}}{{\rm d}t}\rho_{1 \, \e}(t) \leq \|P_{1}(\rho_{M \e}(t),\cdot)\|_{{\rm L}^\infty([0,1])} \, \rho_{1 \, \e}(t)
\end{equation}
from which, using~\eqref{Piub1} and~\eqref{Piub2} with $i=1$, we find that for any $\e>0$ 
$$
\rho_{1 \, \e}(t) \leq \max \left(\|n^0_{1}\|_{{\rm L}^\infty([0,1])}, \frac{A_1}{K} \, \left(2 \, \|a_{1}\|_{{\rm L}^\infty([0,1])}  -1 \right) \right) =: \overline{\rho}_1 \quad \text{for all } \; t \geq 0,
$$
that is, the upper bound~\eqref{eq:UBfinal} on $\rho_{1 \, \e}$ is verified. 

Furthermore, integrating over $[0,1]$ both sides of the IDEs~\eqref{e.mod3} for $n_{2 \, \e}$ and estimating from above using the fact that
$$
Q_{1}(\rho_{M \, \e}(t),x) \leq 2 \, \|p_{1}\|_{{\rm L}^\infty([0,1])} \quad \text{for all } (t,x) \in [0,\infty) \times [0,1] 
$$
along with the upper bound~\eqref{eq:UBfinal} on $\rho_{1 \, \e}$ gives 
\begin{equation}
\label{finubi}
\e \, \frac{{\rm d}}{{\rm d}t}\rho_{2 \, \e}(t) \leq 2 \, \|p_{1}\|_{{\rm L}^\infty([0,1])} \overline{\rho}_1 + \|P_{2}(\rho_{M \e}(t),\cdot)\|_{{\rm L}^\infty([0,1])} \, \rho_{2 \, \e}(t).
\end{equation}
The above differential inequality along with the estimates~\eqref{Piub1} and~\eqref{Piub2} with $i=2$ ensures that there exists $C>0$ such that for any $\e>0$ 
$$
\rho_{2 \, \e}(t) \leq \max \left(\|n^0_{2}\|_{{\rm L}^\infty([0,1])}, \frac{A_2}{K} \, \left(2 \, \|a_{2}\|_{{\rm L}^\infty([0,1])}  -1 \right), \frac{2}{C} \, \|p_{1}\|_{{\rm L}^\infty([0,1])} \, \overline{\rho}_1 \right) =: \overline{\rho}_2
$$
for all $t \geq 0$. Hence, the upper bound~\eqref{eq:UBfinal} on $\rho_{2 \, \e}$ is verified. The upper bounds~\eqref{eq:UBfinal} on $\rho_{i \, \e}$ for $i=3,\ldots,M-1$ can be proved in a similar way using a bootstrap argument. 

Finally, integrating over $[0,1]$ both sides of the IDEs~\eqref{e.mod3} for $n_{M \, \e}$ and estimating from above using the fact that
$$
Q_{M-1}(\rho_{M \, \e}(t),x) \leq 2 \, \|p_{M-1}\|_{{\rm L}^\infty([0,1])} \quad \text{for all } (t,x) \in [0,\infty) \times [0,1] 
$$
along with the upper bound~\eqref{eq:UBfinal} on $\rho_{M-1 \, \e}$ gives the following differential inequality
$$
\e \, \frac{{\rm d}}{{\rm d}t}\rho_{M \, \e}(t) \leq 2 \, \|p_{M-1}\|_{{\rm L}^\infty([0,1])} \overline{\rho}_{M-1} - d \, \rho_{M \e}(t),
$$
which ensures that for any $\e>0$ we have
$$
\rho_{M \, \e}(t) \leq \max \left(\|n^0_{M}\|_{{\rm L}^\infty([0,1])}, \frac{2}{d} \, \|p_{M-1}\|_{{\rm L}^\infty([0,1])} \, \overline{\rho}_{M-1} \right) =: \overline{\rho}_M \quad \text{for all } \; t \geq 0,
$$
that is, the upper bound~\eqref{eq:UBfinal} on $\rho_{M \, \e}$ is verified. 

\begin{acknowledgements}
TS and AM-C were supported by research funding from the German Research Foundation DFG (SFB 873; subproject B08). TL gratefully acknowledges support from the Heidelberg Graduate School (HGS).
\end{acknowledgements}

\bibliography{ATT}   % name your BibTeX data base

%\begin{thebibliography}{}
%%
%% and use \bibitem to create references. Consult the Instructions
%% for authors for reference list style.
%\end{thebibliography}

\end{document}